\documentclass{article} 
\usepackage[margin=3.5cm]{geometry} 
\RequirePackage{natbib}
\newcommand\keywords[1]{\textit{Keywords:} #1}

\usepackage{algorithmic}
\usepackage{amsthm,amscd,amsmath,amssymb,amsfonts,pdfsync,xargs,stmaryrd}
\usepackage{latexsym,bbm,ifthen,twoopt}
\usepackage{color,framed,xargs,amsthm}
\usepackage[draft]{fixme}

\usepackage{aliascnt}
\usepackage{hyperref}
\usepackage{enumitem}

\usepackage{euscript}
\usepackage{amsbsy,amsopn,amstext,amsxtra,enumitem,bbm}

\usepackage{epsf,epsfig,graphics, color}
\usepackage{ float, xargs}
\usepackage[textwidth=3cm, textsize=footnotesize]{todonotes}

\newtheorem{theorem}{Theorem}
\newaliascnt{proposition}{theorem}
\newtheorem{proposition}[proposition]{Proposition}
\aliascntresetthe{proposition}

\newaliascnt{lemma}{theorem}
\newtheorem{lemma}[lemma]{Lemma}
\aliascntresetthe{lemma}

\newaliascnt{corollary}{theorem}
\newtheorem{corollary}[corollary]{Corollary}
\aliascntresetthe{corollary}

\newaliascnt{definition}{theorem}

\aliascntresetthe{definition}

\newaliascnt{example}{theorem}

\aliascntresetthe{example}

\newaliascnt{remark}{theorem}
\newtheorem{remark}[remark]{Remark}
\aliascntresetthe{remark}

\newcommandx\Bbar[4][1=]{\ifthenelse{\equal{#1}{}}{\widetilde B_{#2}^{#4}\langle{Y_{#2} \rangle}}{\widetilde B_{#1}^{#4}\langle{\chunk{Y}{#1}{#2} \rangle}}}
\newcommandx\Cbar[4][1=]{\ifthenelse{\equal{#1}{}}{\widetilde C_{#2}^{#4}\langle{Y_{#2} \rangle}}{\widetilde C_{#1}^{#4}\langle{\chunk{Y}{#1}{#2} \rangle}}}

\newcommandx\Bhat[4][1=]{\widehat B_{#1}^{#4}\langle{Y_{0:#2} \rangle}}
\newcommandx\Chat[4][1=]{\widehat C_{#1}^{#4}\langle{Y_{0:#2} \rangle}}

\usepackage{xspace}
\newcommand\jsd{JSD\@\xspace}
\newcommand\pg{PG\@\xspace}
\newcommand\smc{SMC\@\xspace}

\newcommand\mcmc{MCMC\@\xspace}

\newcommand{\range}[2]{#1, \, \dots, \, #2}      
\newcommand{\prange}[2]{(#1, \, \dots, \, #2)}   

\newcommand\setP{\mathcal{P}(\Xset)}    
\newcommand\kernelPG{P_{T,N}}           
\newcommand\mineps[2]{\epsilon_{#1,#2}} 

\newcounter{hypA'}

\newenvironment{hyp}[1]{
\begin{enumerate}[label=(\textbf{\sf #1}-\arabic*),resume=hyp#1]\begin{sf}}
{\end{sf}\end{enumerate}}

\def\rset{\ensuremath{\mathbb{R}}}

\def\nset{\ensuremath{\mathbb{N}}}
\def\zset{\ensuremath{\mathbb{Z}}}

\renewcommand{\log}{\ln}

\newcommand{\be}{\begin{equation}}     
\newcommand{\ee}{\end{equation}}        

\newcommandx\supnorm[2][1=]{| #2 |^{#1}_\infty}
\newcommandx\lnorm[3][1=]{\lVert #2 \rVert^{#1}_{#3}}
\newcommandx\vnorm[2][1=]{\lVert #2 \rVert^{#1}}
\def\borel{\mathcal{B}}

\newcommand{\coint}[1]{\left[#1\right)}

\newcommand{\ooint}[1]{\left(#1\right)}

\newcommandx\cov[3][1=]{\ensuremath{\mathrm{Cov}_{#1}\left( #2,#3 \right)}}
\newcommandx\var[2][1=]{\ensuremath{\mathrm{Var}_{#1}\left( #2\right)}}
\newcommandx\cvar[4][1=,2=]{\ensuremath{\mathrm{Var}^{#2}_{#1}\left(  #3 \mid #4 \right)}}
\newcommandx\ccov[3][1=]{\ensuremath{\mathrm{Cov}_{#1}\left(  #2 \mid #3 \right)}}

\newcommand{\ci}[4][]%
{%
\ifthenelse{\equal{#1}{}}{\ensuremath{#2 \perp\!\!\!\perp #3 \mid #4 }}{\ensuremath{#2 \perp\!\!\!\perp #3 \mid #4 \; \: [#1]}}%
}

\newcommandx\bprob[2][1=,2=]{\ensuremath{\bar{{\mathbb P}}_{#1}^{#2}}}
\newcommandx\prob[2][1=,2=]{\ensuremath{{\mathbb P}_{#1}^{#2}}}
\newcommandx\besp[2][1=,2=]{\ensuremath{\bar{{\mathbb E}}_{#1}^{#2}}}
\newcommandx\esp[2][1=,2=]{\ensuremath{{\mathbb E}_{#1}^{#2}}}
\newcommandx\espp[3][1=,2=]{\ensuremath{{\mathbb E}_{#1}^{#2} \left[ #3 \right]}}
\newcommandx\cesp[4][1=,2=]{\ensuremath{{\mathbb E}_{#1}^{#2}\left[  #3 \bigm| #4 \right]}}
\newcommandx\cespcheck[4][1=,2=]{\ensuremath{\check{{\mathbb E}}_{#1}^{#2}\left[  #3 \bigm| #4 \right]}}
\newcommandx\cbesp[4][1=,2=]{\ensuremath{\bar{{\mathbb E}}_{#1}^{#2}\left[  #3 \bigm| #4 \right]}}
\newcommandx\cprob[4][1=,2=]{\ensuremath{{\mathbb P}_{#1}^{#2}\left[  #3 \bigm| #4 \right]}}
\newcommandx\cprobcheck[4][1=,2=]{\ensuremath{\check{{\mathbb P}}_{#1}^{#2}\left[  #3 \bigm| #4 \right]}}
\newcommandx\cbprob[4][1=,2=]{\ensuremath{\bar{{\mathbb P}}_{#1}^{#2}\left[  #3 \bigm| #4 \right]}}
\newcommandx\sequence[3][2=t,3=\zset]
{\ifthenelse{\equal{#3}{}}{\ensuremath{\{ #1_{#2}\}}}{\ensuremath{\{ #1_{#2}, \eqsp #2 \in #3 \}}}}

\newcommandx\dsequence[4][3=t,4=\zset]{\ensuremath{\{ (#1_{#3}, #2_{#3}), \eqsp #3 \in #4 \}}}

\newcommandx\proba[3][1=,3=]{{\mathbb P}_{#1}^{#3} (#2)}

\newcommand{\rme}{\mathrm{e}}
\newcommand{\rmd}{\mathrm{d}}

\newcommandx\cespt[4][1=,2=]{\ensuremath{{\mathbb E}_{#1}^{#2}[  #3 \bigm| #4 ]}}

\newcommandx{\plone}[1][1=\zset]{{\ensuremath{\mathrm{\ell}^1}(#1)}}

\newcommand{\1}{\ensuremath{\mathbbm{1}}}

\newcommand{\wrt}{with respect to}

\newcommand{\rhs}{right-hand side}

\newcommand{\iid}{i.i.d.}
\newcommandx{\as}[1][1=\PP]{\ensuremath{#1\-\mathrm{a.s.}}}

\newcommand{\ie}{i.e.}
\newcommand{\eg}{e.g.}

\newcommand{\eqsp}{\;}
\newcommand{\eqdef}{\ensuremath{:=}}

\def\mcf{\mathcal{F}}
\newcommand{\mcff}[2]{\mathcal{F}_{#1}^{#2}}

\newcommand{\PP}[1][]{\ifthenelse{\equal{#1}{}}{\ensuremath{\mathbb{P}}}{\ensuremath{\mathbb{P}\left[ #1 \right]}}}
\newcommand{\PPstat}[1][]{\ifthenelse{\equal{#1}{}}{\ensuremath{\bar{\mathbb{P}}}}{\ensuremath{\bar{ \mathbb{P}}\left[ #1 \right]}}}

\newcommand{\PPup}[2][1=]{\ifthenelse{\equal{#1}{}}{\ensuremath{\mathbb{P}}}{\ensuremath{\mathbb{P}^{#1}\left[ #2 \right]}}}
\newcommand{\PPdoup}[3][1=,2=]{
\ensuremath{\mathbb{P}_{#1}^{#2}\left[ #3 \right]}}
\newcommandx{\PE}[3][1=,3=]{\ifthenelse{\equal{#2}{}}{\ensuremath{\mathbb{E}^{#1}_{#3}}}{\ensuremath{{\mathbb{E}^{#1}_{#3}}\left[ #2 \right]}}}
\newcommandx{\PEstat}[3][1=,3=]{\ifthenelse{\equal{#2}{}}{\ensuremath{\bar{\mathbb{E}}^{#1}_{#3}}}{\ensuremath{{\bar{\mathbb{E}}^{#1}_{#3}}\left[ #2 \right]}}}

\newcommand{\CPE}[3][]
{\ifthenelse{\equal{#1}{}}{{\mathbb E}\left[\left. #2 \, \right| #3 \right]}{{\mathbb E}_{#1}\left[\left. #2 \, \right | #3 \right]}}
\newcommand{\CPEup}[3][]
{\ifthenelse{\equal{#1}{}}{{\mathbb E}\left[\left. #2 \, \right| #3 \right]}{{\mathbb E}^{#1}\left[\left. #2 \, \right | #3 \right]}}
\newcommand{\CPEdoup}[4][1=,2=]
{{\mathbb E}_{#1}^{#2}\left[\left. #3 \, \right| #4 \right]}
\newcommand{\CPP}[3][]
{\ifthenelse{\equal{#1}{}}{{\mathbb P}\left(\left. #2 \, \right| #3 \right)}{{\mathbb P}_{#1}\left(\left. #2 \, \right | #3 \right)}}
\newcommand{\CPPup}[3][]
{\ifthenelse{\equal{#1}{}}{{\mathbb P}\left[\left. #2 \, \right| #3 \right]}{{\mathbb P}^{#1}\left[\left. #2 \, \right | #3 \right]}}
\newcommand{\CPPdoup}[4][1=,2=]
{{\mathbb P}_{#1}^{#2}\left[\left. #3 \, \right| #4 \right]}
\newcommandx{\PVar}[1][1=]{\ensuremath{\operatorname{Var}_{#1}}}

\newcommandx{\PCov}[1][1=]{\ensuremath{\operatorname{Cov}}_{#1}}

\newcommandx{\VnormFunc}[3][1=]{\ensuremath{\left|#2\right|^{#1}_{\mathrm{#3}}}}

\newcommandx{\oscnorm}[3][1=,3=]%
{\operatorname{osc}^{#1}_{#3}\left(#2\right)}

\newcommandx{\Vdist}[3][1=V]{d_{#1}(#2, #3)}
\newcommandx\functionset[2][1=]{\mathbb{F}_{#1}(\mathsf{#2},\mathcal{#2})}
\newcommandx\functionsetspec[1][1=]{\mathbb{F}_{#1}}

\newcommandx\measureset[2][1=]{\mathbb{M}_{#1}(\mathcal{#2})}
\newcommandx\measuresetspec[2][1=]{\mathbb{M}_{#1}({#2})}
\newcommandx\ball[3][1=]{\mathrm{B}^{#1} (#2,#3)}
\newcommandx{\wassersym}[1][1=d]{\mathbf{W}_{#1}}
\newcommandx{\wasser}[3][1=d]{\mathbf{W}_{#1}\left(#2,#3\right)}
\newcommandx{\prohosym}[1][1=d]{\boldsymbol{\rho}_{#1}}
\newcommandx{\proho}[3][1=d]{\boldsymbol{\rho}_{#1}\left(#2,#3\right)}
\newcommandx{\dualblsym}[1][1=d]{\boldsymbol{\beta}_{#1}}
\newcommandx{\dualbl}[4][1=d,2=]{\boldsymbol{\beta}^{#2}_{#1}\left(#3,#4\right)}
\newcommandx{\Vdistance}[1][1=V]{d_{#1}}

\newcommandx{\aslim}[1]{\ensuremath{\stackrel{#1\-\text{a.s.}}{\longrightarrow}}}
\newcommandx{\plim}[1]{\ensuremath{\stackrel{#1}{\longrightarrow}}}
\newcommandx{\dlim}[1]{\ensuremath{\stackrel{#1}{\Longrightarrow}}}

 \newcommandx{\taslim}[1]{\ensuremath{\rightarrow_{#1\-\text{a.s.}}}}
 \newcommandx{\tplim}[1]{\ensuremath{\rightarrow_{#1}}}
 \newcommandx{\tdlim}[1]{\ensuremath{\Rightarrow_{#1}}}

\newcommand{\Xset}{\ensuremath{\mathsf{X}}}

\newcommand{\Xsigma}{\ensuremath{\mathcal{X}}}

\newcommand{\Yset}{\ensuremath{\mathsf{Y}}}
\newcommand{\Ysigma}{\ensuremath{\mathcal{Y}}}

\newcommand{\Zset}{\ensuremath{\mathsf{Z}}}
\newcommand{\Zsigma}{\ensuremath{\mathcal{Z}}}

\newcommand{\chunk}[4][]%
{\ifthenelse{\equal{#1}{}}{\ensuremath{{#2}_{#3:#4}}}{\ensuremath{#2^{#1}}_{#3:#4}}
}

\newcommandx{\pred}[4][1=,4=]%
{%
\ifthenelse{\equal{#1}{}}{\ensuremath{\phi^{#4}_{#2|#3}}}{\ensuremath{\phi^{#4}_{#1,#2|#3}}}%
}

\newcommandx{\adjfunc}[4][1=,3=,4=]
{\ifthenelse{\equal{#1}{}}{\ifthenelse{\equal{#3}{}}{\vartheta^{#4}_{#2}}{\vartheta^{#4}_{#2}(#3)}}
{\ifthenelse{\equal{#1}{smooth}}{\ifthenelse{\equal{#3}{}}{\tilde{\vartheta}^{#4}_{#2}}{\tilde{\vartheta}^{#4}_{#2}(#3)}}
{\ifthenelse{\equal{#1}{fully}}{\ifthenelse{\equal{#3}{}}{\vartheta^\star_{#2}}{\vartheta^\star_{#2}(#3)}}{\mathrm{erreur}}}}}

\newcommandx{\post}[4][1=,4=]%
{
\ifthenelse{\equal{#1}{}}
{
\ifthenelse{\equal{#4}{}}{\ensuremath{\phi_{#2}\langle #3 \rangle}}{\ensuremath{\phi_{#2}^{#4}\langle #3 \rangle}}
}%
{\ifthenelse{\equal{#1}{hat}}
{
\ifthenelse{\equal{#4}{}}{\ensuremath{\phi^N_{#2}\langle #3 \rangle}}{\ensuremath{\phi_{#2}^{N,#4}\langle #3 \rangle}}
}
{\ifthenelse{\equal{#1}{tilde}}{\ensuremath{\tilde{\phi}^{N}_{#2}\langle #3 \rangle}}}
}
}

\newcommandx{\filt}[3][1=,3=]%
{
\ifthenelse{\equal{#1}{}}
     {\ensuremath{\phi^{#3}_{#2}}}%
     {\ifthenelse{\equal{#1}{hat}}{\ifthenelse{\equal{#3}{}}{\ensuremath{\phi^{N}_{#2}}}{\ensuremath{\phi^{N,#3}_{#2}}}}
{\ifthenelse{\equal{#1}{tilde}}{\ifthenelse{\equal{#3}{}}{\ensuremath{\tilde{\phi}^{N}_{#2}}}{\ensuremath{\tilde{\phi}^{N,#3}_{#2}}}}
{\ifthenelse{\equal{#1}{tar}}{\ensuremath{\phi^{N,\mathrm{tar}}_{#2}}}
{\ifthenelse{\equal{#1}{aux}}{\ensuremath{\phi^{N,\mathrm{aux}}_{#2}}}
{\ifthenelse{\equal{#1}{prop}}{\ensuremath{\phi^{N,\mathrm{prop}}_{#2}}}
{\ensuremath{\phi^{#1}_{#2}}
}
}
}
}
}
}
}

\newcommandx{\unfilt}[3][1=,3=]%
{
\ifthenelse{\equal{#1}{}}
     {\ensuremath{\gamma^{#3}_{#2}}}%
     {\ifthenelse{\equal{#1}{hat}}{\ifthenelse{\equal{#3}{}}{\ensuremath{\kappa^{N}_{#2}}}{\ensuremath{\kappas^{N,#3}_{#2}}}}
{\ifthenelse{\equal{#1}{tilde}}{\ifthenelse{\equal{#3}{}}{\ensuremath{\tilde{\gamma}^{N}_{#2}}}{\ensuremath{\tilde{\gamma}^{N,#3}_{#2}}}}
{\ifthenelse{\equal{#1}{tar}}{\ensuremath{\gamma^{N,\mathrm{tar}}_{#2}}}
{\ifthenelse{\equal{#1}{aux}}{\ensuremath{\gamma^{N,\mathrm{aux}}_{#2}}}
{\ifthenelse{\equal{#1}{prop}}{\ensuremath{\gamma^{N,\mathrm{prop}}_{#2}}}
{\ensuremath{\gamma^{#1}_{#2}}
}
}
}
}
}
}
}

\newcommandx{\normLip}[2][1=d]{\left|#2\right|_{\mathrm{L}(#1)}}
\newcommandx{\normBL}[2][1=d]{\left|#2\right|_{\mathrm{BL}(#1)}}
\newcommandx{\dualnormBL}[2][1=d]{\|#2\|_{\mathrm{BL}(#1)}}
\newcommandx{\duallipsym}[1][1=d]{\boldsymbol{\gamma}_{#1}}
\newcommandx{\duallip}[4][1=d,2=]{\boldsymbol{\gamma}^{#2}_{#1}\left(#3,#4\right)}

\newcommandx{\wass}[2][1=d]{\lVert #2\rVert_{\mathrm{L}(#1)}}

\newcommandx{\dobru}[2][1=\mathrm{TV}]{\dobrush_{#1}\left( #2\right)}
\newcommand{\dobrush}{\Delta}

\def\epartsymb{X}

\newcommand{\etpart}[2][]
{%
\ifthenelse{\equal{#1}{}}{\ensuremath{\tilde{\epartsymb}^{#2}}}{\ensuremath{\tilde{\epartsymb}^{#2}_{#1}}}
}

\newcommand{\Xinitis}[2]{r_0\langle #1 \rangle(#2)}
\newcommand{\Xinitisf}[1]{r_0\langle #1 \rangle}
\newcommand{\Xinit}[1]{\mu(#1)}
\newcommand{\Xinitv}{\mu}
\newcommand{\ewght}[2]{\omega_{#1}^{#2}}
\newcommand{\epart}[2]{X_{#1}^{#2}}
\newcommandx{\ewghtfunc}[4][1=]{w^{#1}\langle #2 \rangle(#3, #4)}
\newcommandx{\ewghtfuncf}[2][1=]{w^{#1}\langle #2 \rangle}
\newcommandx{\ewghtfuncfInit}[2][1=]{w^{#1}_0\langle #2 \rangle}
\newcommandx{\ewghtfuncz}[3][1=]{w^{#1}\langle #2 \rangle(#3)}
\newcommandx{\kun}[4][1=]{q^{#1}\langle #2 \rangle(#3, #4)}
\newcommandx\Kun[4][1=]{Q^{#1}\langle #2 \rangle(#3, #4)}
\newcommandx\Kunf[2][1=]{Q^{#1}\langle #2 \rangle}
\newcommandx{\kis}[4][1=]{r^{#1}\langle #2 \rangle(#3, #4)}
\newcommandx{\Kis}[4][1=]{R^{#1}\langle #2 \rangle(#3, #4)}
\newcommandx\Kisf[2][1=]{R^{#1}\langle #2 \rangle}

\newcommand{\logl}[2][]%
{%
\ifthenelse{\equal{#1}{}}{\ensuremath{\ell_{#2}}}{\ensuremath{\ell_{#1,#2}}}%
}
\newcommand{\lnl}[2][]%
{%
\ifthenelse{\equal{#1}{}}{\ensuremath{\ell_{#2}}}{\ensuremath{\ell_{#1,#2}}}%
}

\newcommand{\VRoot}{\ensuremath{S}}
\newcommand{\VCov}[1][]%
{%
\ifthenelse{\equal{#1}{}}{\VRoot \VRoot^t}{\VRoot_{#1} \VRoot^t_{#1}}%
}

\newcommandx{\mcfpart}[2][1=]{%
\ifthenelse{\equal{#1}{}}{{\mathcal{F}^{#2}}}{\mathcal{F}_{#1}^{#2}}%
}

\newcommand{\tsumweight}[2][]{%
\ifthenelse{\equal{#1}{}}{\ensuremath{\widetilde{\Omega}^{#2}}}{\ensuremath{\widetilde{\Omega}_{#1}^{#2}}}}

\newcommandx{\BK}[2][2=]{B^{#2}_{#1}}
\newcommandx{\bk}[2][2=]{b^{#2}_{#1}}

\newcommand{\esssup}[2][]%
{\ifthenelse{\equal{#1}{}}{| #2 |_\infty}{| #2 |^{#1}_{\infty}}}
\newcommandx{\F}[2][2=]{F_{#1}^{#2}}

\newcommand{\asymVar}[4][]{
\ifthenelse{\equal{#1}{}}{\ifthenelse{\equal{#4}{}}{\ensuremath{\Gamma_{#2|#3}}}{\ensuremath{\Gamma_{#2|#3}\left[#4\right]}}}
{\ifthenelse{\equal{#4}{}}{\ensuremath{\Gamma_{#1,#2|#3}}}{\ensuremath{\Gamma_{#1,#2|#3}\left[#4\right]}}}
}
\newcommand{\incrasymVar}[4][]{
\ifthenelse{\equal{#1}{}}{\ensuremath{\sigma^2_{#2,#3}\left[#4\right]}}{\ensuremath{\sigma^2_{#1,#2,#3}\left[#4\right]}}
}

\def\param{\theta}

                                       \newcommand{\Param}{\ensuremath{\Theta}}

\newcommand{\tparam}{\ensuremath{{\theta_\star}}}
\newcommand{\tpi}{\ensuremath{{\pi^{\tparam}}}}

\newcommandx{\m}[1][1=]
{\ifthenelse{\equal{#1}{}}{\ensuremath{m}}{\ensuremath{m^{#1}}}}
\newcommandx{\M}[1][1=]
{\ifthenelse{\equal{#1}{}}{\ensuremath{M}}{\ensuremath{M^{#1}}}}

\newcommandx{\g}[1][1=]
{\ifthenelse{\equal{#1}{}}{\ensuremath{g}}{\ensuremath{g^{#1}}}}

\newcommandx{\lhood}[3][1=]%
{\operatorname{L}_{#1}({#2};\, {#3})
}

\newcommandx{\hlhood}[4][1=,4=]%
{\widehat{\operatorname{L}}^{#4}_{#1}({#2};\, {#3})
}

\newcommandx{\cdens}[4][1=,4=]
{p^{#4}_{#1}(#2|#3)}
\newcommandx{\cdensstat}[4][1=,4=]
{\bar{p}^{\thinspace #4}_{#1}(#2|#3)}

\newcommandx{\dens}[3][1=,3=]
{p^{#3}_{#1}(#2)}
\newcommandx{\densstat}[3][1=,3=]
{\bar{p}^{\thinspace #3}_{#1}(#2)}

\newcommandx{\hdens}[3][1=,2=]%
{
\ifthenelse{\equal{#1}{}}
{\ifthenelse{\equal{#2}{}}{\hat{p}(#3)}{\hat{p}^{#2}(#3)}}
{\ifthenelse{\equal{#2}{}}{\hat{p}_{#1}(#3)}{\hat{p}_{#1}^{#2}(#3)}}
}

\newcommandx{\deriv}[1][1=]{\nabla_{#1}}

\newcommandx\lkdM[3][1=,3=]{
\ifthenelse{\equal{#2}{}}
{ \mathsf{L}_{#1}^{#3}}
{ \mathsf{L}_{#1}^{#3}(#2)}
}
\newcommandx\lkdMStat[3][1=,3=]{
\ifthenelse{\equal{#2}{}}
{ \bar{\mathsf{L}}_{#1}^{#3}}
{ \bar{\mathsf{L}}_{#1}^{#3}(#2)}
}

\renewcommand{\-}{\mbox{-}}

\newcommand{\CPEu}[3][]
{\ifthenelse{\equal{#1}{}}{{\mathbb E}\left[\left. #2 \, \right| #3 \right]}{{\mathbb E}^{#1}\left[\left. #2 \, \right | #3 \right]}}

\newcommandx{\f}[2][1=\theta]{f^{#1}_{#2}}

\def\bigone{\mathbf{1}}

\title{Uniform ergodicity of the Particle Gibbs sampler}

\author{
  Fredrik Lindsten\\
  Division of Automatic Control\\
  Link{\"o}ping University, Link\"oping. Sweden\\
  \and
  Randal Douc \\
  Department  CITI, Institut Mines-Telecom/CNRS UMR 5157 \\
  Telecom Sudparis, Evry, France\\
  \and
  Eric Moulines \\
  Department  LTCI, Institut Mines-Telecom/CNRS UMR 5141 \\
  Telecom Paristech, Paris, France\\
}

\date{January 23, 2014}

\begin{document}
\maketitle

\begin{abstract}
The particle Gibbs (\pg) sampler is a systematic way of using a particle filter within Markov chain Monte Carlo (MCMC).
This results in an off-the-shelf Markov kernel on the space of state trajectories, which can be used to simulate from the
full joint smoothing distribution for a state space model in an MCMC scheme.
We show that the PG Markov kernel is uniformly ergodic under rather general assumptions, that we will carefully review and discuss.
In particular, we provide an explicit rate of convergence which reveals that:
\emph{(i)} for fixed number of data points, the convergence rate can be made arbitrarily good by increasing the number of particles, and
\emph{(ii)} under general mixing assumptions, the convergence rate can be kept constant by increasing the number of particles
superlinearly with the number of observations. We illustrate the applicability of our result by studying in detail
two common state space models with non-compact state spaces. \par
\vspace{1ex}
\keywords{Particle Gibbs, Particle Markov chain Monte Carlo, Conditional sequential Monte Carlo, Particle smoothing, State space models}
\end{abstract}

\section{Introduction}
Statistical inference in general state space hidden Markov models involves computation 
of the posterior distribution of a set
 $\chunk{X}{t}{t'} \eqdef [X_t, \dots, X_{t'}]$ of hidden state variables conditionally on a record $\chunk{Y}{0}{T}$ of observations,
which we denote as $\post{t:t'}{\chunk{Y}{0}{T}}$.
Of particular interest is the so called \emph{joint smoothing distribution} (\jsd) $\post{0:T}{\chunk{Y}{0}{T}}$.
Any marginal or fixed-interval smoothing distribution can be obtained from the \jsd by marginalization.
The \jsd can be expressed in closed-form only in very specific cases, principally, when the state space model is linear and Gaussian or when the state space of the hidden Markov chain is a finite set. In the vast majority of cases, nonlinearity or non-Gaussianity render analytic solutions intractable.

This limitation has lead to an increase of interest in computational strategies handling more general state and measurement equations. Among these, \emph{sequential {M}onte {C}arlo} (SMC) methods play a central role. SMC methods---in which the \emph{sequential importance sampling} and \emph{sampling importance resampling} methods proposed by \citet{HandschinM:1969} and \citet{Rubin:1987}, respectively, are combined---refer to a class of algorithms approximating a sequence of probability distributions, defined on a sequence of probability spaces. This is done by updating recursively a set of random \emph{particles} with associated nonnegative importance weights. The SMC methodology has emerged as a key tool for approximating \jsd flows in general state space models; see \citet{delmoral:2004,DelMoralD:2009,DoucetJ:2011} for general introductions as well as applications and theoretical results for SMC methods.

However, a well known problem with SMC methods is that the particle approximation of any marginal smoothing distribution
$\post{t:t}{\chunk{Y}{0}{T}}$ becomes inaccurate for $t \ll T$. The reason is that the particle trajectories degenerate gradually as the interacting particle
system evolves \citep{GodsillDW:2004,FearnheadWT:2010}.
To address this problem, several methods have been proposed; see \citet{LindstenS:2013} and the references therein.
Among these methods, the recently introduced particle Markov chain Monte Carlo (PMCMC) framework,
proposed in the seminal paper by \citet{AndrieuDH:2010}, plays a prominent role.
PMCMC samplers make use of \smc (or variants thereof) to construct efficient, high-dimensional \mcmc kernels which are reversible \wrt\ the \jsd.
These methods can then be used as components of more general sampling schemes relying on Markov kernels,
for instance enabling joint state and parameter inference in general state space models.
We will not discuss such composite sampling schemes in this paper, but instead focus on one of the PMCMC kernels that can be used to simulate from the \jsd.

Coupling SMC and MCMC is very useful since the distribution of the state sequence given the stream of observations
is generally both high-dimensional and strongly dependent, rendering the design of alternative \mcmc procedures,
such as single-state 
Gibbs samplers and Metropolis-Hastings samplers, problematic.
PMCMC has already found many applications in areas such as hydrology \citep{VrugtBDS:2013},
finance \citep{PittSGK:2012}, systems biology \citep{GolightlyW:2011}, and epidemiology \citep{RasmussenRK:2011}, to mention a few.
Several methodological developments of the framework have also been made; see \eg\ \citet{WhiteleyAD:2010,LindstenJS:2012,ChopinS:2013,PittSGK:2012}.

PMCMC algorithms can, broadly speaking, be grouped into two classes of methods: those based on particle independent Metropolis-Hastings (PIMH) kernels and those based on particle Gibbs (\pg) kernels. The two classes of kernels are motivated in different ways and they have quite different
properties. The former class, PIMH, exploits the fact that the \smc method defines an unbiased
estimator of the likelihood, which is used in place of the intractable likelihood in the MH acceptance probability.
This method can thus be viewed as a special case of the pseudo-marginal method introduced by \citet{beaumont:2003,AndrieuR:2009} and later analyzed by
\citet{AndrieuV:2012,lee:latuszynski:2012}.
The latter class, \pg, on the other hand relies on
conditioning the underlying \smc sampler on a reference trajectory to enforce the correct limiting distribution of the kernel; see \autoref{sec:pg}.
This algorithm can be interpreted as a Gibbs sampler for an extended model where the random variables generated by the \smc sampler are treated as auxiliary variables.

One of the main practical issues with PMCMC algorithms is the choice of the number, $N$, of  particles.
Using fewer particles will result in faster computations at each iteration, but can at the same time
result in slower mixing of the resulting Markov kernel.
For a fixed computational budget, there is a trade-off between
taking the number of particles $N$ large to get a faster mixing kernel, and to run many iterations of the MCMC sampler.
\citet{AndrieuR:2009,AndrieuV:2012,lee:latuszynski:2012} investigate 
the rate of convergence of the pseudo-marginal method and characterize the approximation of the marginal algorithm by the pseudo-marginal algorithm
in terms of the variability of their respective ergodic averages.
\citet{DoucetPK:2012} and \citet{PittSGK:2012} conclude, using partially heuristic arguments,
that it is close to optimal to let $N$ scale at least linearly with $T$.

The theoretical properties of the \pg kernel, however, are not as well understood. 
\citet{AndrieuDH:2010} establish under weak conditions that the \pg kernel is $\phi$-irreducible and aperiodic for any $N\geq 2$
(see \citet{MeynT:2009} for definitions).
However, this does not provide a control for the rate of convergence of the iterates of the \pg kernel to stationarity.
In this work, we establish that the \pg kernel is, under mild assumptions, uniformly ergodic. This interesting property has already been established in an earlier
work by \citet{ChopinS:2013}, but we give here a more straightforward proof under weaker conditions, which in addition
provides an explicit lower bound for the convergence rate.

During the preparation of this manuscript, a preprint was made available by \citet{AndrieuLV:2013}, who, independently, have found similar results as presented
here. Indeed, they establish basically the same lower bound on the minorizing constant for the \pg kernel
(which they refer to as the iterated conditional SMC kernel), though using a different proof technique based on
a ``doubly conditional'' SMC algorithm.
There are, however, several differences between these two contributions. We focus in particular on analyzing
the minorizing constant under mixing conditions for the state space model
which hold very generally, even if the state space is not compact (see \autoref{sec:moment-assumption}).
We then study how the number of particles $N$ should be increased
with the number of observations $T$. We show that under weak assumptions,
it suffices to increase the number of particles $N$ as  $T^\delta$ where $\delta\geq 1$ can
be determined explicitly. 
This is in contrast with \citet{AndrieuLV:2013} who, effectively, assume a compact state space; see \autoref{rem:comparison-alv} and \autoref{sec:strong-mixing}.
On the other hand, \citet{AndrieuLV:2013} study necessary (\ie, not only sufficient) conditions for uniform ergodicity
and translate the convergence results for the \pg kernel to a composite \mcmc scheme for simulating
both states and parameters of a state space model.
Given these differences, we believe that the two contributions complement each other.

This paper is organized as follows: In \autoref{sec:notation-prob-form} we introduce our notation,
and in \autoref{sec:pg} we review the \pg sampler and
formally define the \pg Markov kernel. In \autoref{sec:main-result} we state the main results,
starting with a minorization condition for the \pg kernel followed by mixing conditions that allow for
time uniform control of the convergence rate.
In \autoref{sec:examples} we study, in detail, two commonly used
state space models (with non-compact state spaces) to illustrate how the conditions of our results can be verified
in practice.
The proofs of the main theorems are postponed to Sections~\ref{sec:proof}~and~\ref{sec:proof:prop:N-depend-on-T}.

\section{Notations and problem statement}\label{sec:notation-prob-form}%
Let $(\Xset,\Xsigma)$ and $(\Yset,\Ysigma)$ be two measurable spaces
and let $\setP$ be the set of all probability measures on $(\Xset,\Xsigma)$.
Let $M$ be a kernel on $(\Xset,\Xsigma)$ and $G$ a kernel on $(\Xset,\Ysigma)$. Assume that for all $x\in \Xset$, $G(x,\cdot)$ is dominated by some common nonnegative measure $\kappa$ on $(\Yset,\Ysigma)$ and denote by $g(x,\cdot)$ its Radon-Nikodym derivative, \ie, for all $(x,y) \in \Xset \times \Yset$,
$$
g(x,y)=\frac{\rmd G(x,\cdot)}{\rmd \kappa(\cdot)}(y)\eqsp.
$$
Let $\{(X_t,Y_t)\,, t \in \nset\}$ be a hidden Markov chain associated to the pair $(M,G)$. That is, $\{(X_t,Y_t)\,, t \in \nset\}$ is a Markov chain with transition kernel defined by: for all $(x,y) \in \Xset \times \Yset$ and all $C \in \Xsigma\otimes \Ysigma$,
$$
((x,y), C) \mapsto \iint_{C} M(x,\rmd x') G(x',\rmd y')\eqsp.
$$
The sequence \sequence{X}[t][\nset] is usually not observed and inference should be carry out on the basis of the observations \sequence{Y}[t][\nset] only.
With $\Xinitv \in \setP$ being the initial distribution of the hidden state process, for all $t\geq 0$, denote by
$$
\chunk{y}{0}{t} \mapsto \dens[\Xinitv]{\chunk{y}{0}{t}}\eqdef \int \Xinitv(\rmd x_0) g(x_0,y_0) \prod_{s=1}^t M(x_{s-1},\rmd x_s) g(x_s,y_s) \eqsp,
$$
the density of the observations $\chunk{Y}{0}{t}$ with respect to $\kappa^{\otimes (t+1)}$.
In what follows, we set, by abuse of notation, for all $x \in\Xset$,
\begin{equation}
\label{eq:def:p-delta}
\dens[x]{\chunk{y}{0}{t}}=\dens[\delta_x]{\chunk{y}{0}{t}} \eqsp,
\end{equation}
where $\delta_x$ is the Dirac measure at $x$.

For all $y  \in \Yset$, define the (unnormalized) kernel $\Kunf{y}$ on $(\Xset,\Xsigma)$ by
\begin{equation}
\label{eq:def-Kun}
\Kun{y}{x}{A}= \int M(x,\rmd x') g(x',y) \1_A(x') \eqsp,
\end{equation}
and for all $s \leq t$ and all $\chunk{y}{s}{t} \in\Yset^{t-s+1}$, define the kernel $\Kunf{\chunk{y}{s}{t}}$ on $(\Xset,\Xsigma)$ by
\begin{equation}
\label{eq:def-chunk-Kun}
\Kun{\chunk{y}{s}{t}}{x}{A}= \Kunf{y_s} \Kunf{y_{s+1}} \dots \Kun{y_t}{x}{A} \eqsp.
\end{equation}
In what follows, we set by convention $\Kun{\chunk{y}{s}{t}}{x}{A}= 1$ for all $s>t$.
With these notations, $\dens[\Xinitv]{\chunk{y}{0}{t}} = \Xinitv \Kunf{\chunk{y}{0}{t}}\bigone$ where $\bigone$ is the constant function, $\bigone(x) = 1$ for all $x \in \Xset$.
For all $\Xinitv\in\setP$ and for all $0\leq s\leq t$, denote
\begin{align*}
& \cdens[\Xinitv]{\chunk{y}{s}{t}}{\chunk{y}{0}{s-1}} \eqdef \begin{cases}
 \dens[\Xinitv]{\chunk{y}{0}{t}}/ \dens[\Xinitv]{\chunk{y}{0}{s-1}}\,, &\mbox{if } \dens[\Xinitv]{\chunk{y}{0}{s-1}} \neq 0\eqsp,\\
 0\,, &\mbox{otherwise}
 \end{cases}
\end{align*}
with the convention $\cdens[\Xinitv]{\chunk{y}{0}{t}}{\chunk{y}{0}{-1}}=\dens[\Xinitv]{\chunk{y}{0}{t}}$.

A quantity of central interest is the \jsd, given by
\begin{align}
  \label{eq:jsd-def}
   \post{\Xinitv,0:t}{\chunk{y}{0}{t}}(D) \eqdef \frac{1}{\dens[\Xinitv]{\chunk{y}{0}{t}}} \int \Xinitv(\rmd x_0) g(x_0,y_0) \prod_{s=1}^t M(x_{s-1},\rmd x_s) g(x_s,y_s) \1_D(x_{0:t}) \eqsp,
\end{align}
for all $D\in\Xsigma^{\otimes (t+1)}$.
With $T$ being some final time point, the \pg sampler (reviewed in the subsequent section) defines a Markov kernel which is reversible \wrt\ $\post{\Xinitv,0:T}{\chunk{y}{0}{T}}$. Samples drawn from the \pg kernel can thus be used to draw inference
about the states (and/or parameters) of the state space model.


\section{The particle Gibbs sampler}\label{sec:pg}%
Consider first an \smc sampler targeting the sequence of \jsd{s} defined in~\eqref{eq:jsd-def}.
The \smc sampler approximates $\post{\Xinitv,0:t}{\chunk{Y}{0}{t}}$ by a collection of weighted samples
$\{(\epart{0:t}{i}, \ewght{t}{i})\}_{i=1}^N$, in the sense that
\begin{align*}
  \post[hat]{\Xinitv,0:t}{\chunk{Y}{0}{T}}(h) \eqdef \sum_{i=1}^N \frac{ \ewght{t}{i} }{ \sum_{\ell=1}^N \ewght{t}{\ell} } h(\epart{0:t}{i})
\end{align*}
is an estimator of $\post{\Xinitv,0:t}{\chunk{Y}{0}{t}}(h)$ for a measurable function $h: \Xset^{t+1}\to \rset$. These
weighted samples can be generated in several different ways, see \eg\ \citet{DoucetGA:2000,delmoral:2004,DoucetJ:2011,CappeMR:2005}
and the references therein. Here we review a basic method, though it should be noted that the \pg sampler
can be generalized to more advanced procedures, see \citet{AndrieuDH:2010,ChopinS:2013}.

Initially, $\post{\Xinitv,0:0}{Y_0}$ is approximated by importance sampling. That is, we simulate independently $\{\epart{0}{i}\}_{i=1}^N$
from a proposal distribution: $\epart{0}{i} \sim \Xinitis{Y_0}{\cdot}$. The samples, commonly referred to as \emph{particles}, are then
assigned importance weights,
\begin{align}
  \label{eq:particle-weight-0}
  \ewght{0}{i}=   \ewghtfuncfInit{Y_0}(\epart{0}{i}) \eqsp,
\end{align}
where  $\ewghtfuncfInit{Y_0}(x)=g(x,Y_0) \frac{ \rmd \Xinitv }{\rmd \Xinitisf{Y_0}} (x)$, provided that $\Xinitisf{Y_0}$ is such that $\Xinitv \ll \Xinitisf{Y_0}$.

We proceed inductively.
Denote by $\mcff{t}{N}$ the filtration generated by the particles and weights up to the current time instant $t$:
\begin{equation}
\label{eq:particle-filtration}
  \mcff{t}{N} \eqdef \sigma\left( \{(\epart{0:s}{i}, \ewght{s}{i})\}_{i=1}^N, 0 \leq s \leq t \right).
\end{equation}
Assume that we have at hand a weighted sample $\{(\epart{0:t-1}{i}, \ewght{t-1}{i})\}_{i=1}^N$
approximating the \jsd $\post{\Xinitv,0:t-1}{\chunk{Y}{0}{t-1}}$ at time $t-1$. This weighted sample is then propagated sequentially \emph{forward in time}. This is done by sampling,
conditionally independently given the particle history $\mcff{t-1}{N}$, for each particle $i \in \{1,\ldots, N\}$ an \emph{ancestor index} $A_t^i$ with probability
\begin{equation}
  \label{eq:ancestor-t}
  \CPP{A_t^i = j}{\mcff{t-1}{N}} = \frac{ \ewght{t-1}{j} }{ \sum_{\ell=1}^N \ewght{t-1}{\ell}} \eqsp, \quad j \in \{1,\dots,N\} \eqsp,
\end{equation}
 and then by sampling a new particle position from the proposal kernel $\Kisf{Y_t}$:
\begin{align}
  \label{eq:particle-t}
  \epart{t}{i} \sim \Kis{Y_t}{\epart{t-1}{A_t^i}}{ \cdot }\eqsp.
\end{align}
The particle trajectories (\ie, the ancestral paths of the particles $\epart{t}{i}$, $i \in \{1,\dots,N\}$)
are constructed sequentially by associating the current particle $\epart{t}{i}$ with the particle trajectory of its ancestor:
\begin{equation}
\label{eq:ancestral-path}
\epart{0:t}{i} \eqdef (  \epart{0:t-1}{A_t^i}, \epart{t}{i} )\eqsp.
\end{equation}
Finally, similarly to \eqref{eq:particle-weight-0}
the particles are assigned importance weights given by
\begin{align}
  \label{eq:particle-weight-t}
\ewght{t}{i}= \ewghtfunc{Y_t}{\epart{t-1}{A_t^i}}{\epart{t}{i}}
\eqdef \frac{\rmd \Kun{Y_t}{\epart{t-1}{A_t^i}}{\cdot}}{\rmd \Kis{Y_t}{\epart{t-1}{A_t^i}}{\cdot}} (\epart{t}{i})\eqsp,
\end{align}
where $\Kunf{y}$ is defined in \eqref{eq:def-Kun} and, as before, it is assumed that
$ \Kun{y}{x}{\cdot} \ll \Kis{y}{x}{\cdot}$. This results in a weighted particle system
$\{(\epart{0:t}{i}, \ewght{t}{i})\}_{i=1}^N$ targeting $\post{\Xinitv,0:t}{\chunk{Y}{0}{t}}$, completing the induction. Two classical choices for the proposal kernel $\Kisf{y}$ are:
\begin{equation} \label{eq:bootstrap-fully-adapted}
\Kis{y}{x}{\rmd x'}=
\begin{cases}
M(x,\rmd x') & \mbox{bootstrap filter,}\\
\frac{M(x,\rmd x') g(x',y)}{\int M(x,\rmd x') g(x',y)} & \mbox{fully-adapted filter.}
\end{cases}
\end{equation}
Assume now that $T$ is some final time point and that we are interested in simulating
from the \jsd $\post{\Xinitv,0:T}{\chunk{Y}{0}{T}}$ using an \mcmc procedure.
For that purpose, it is required to define a Harris positive recurrent Markov kernel on the path space $(\Xset^{T+1},\Xsigma^{\otimes (T+1)})$ having the \jsd $\post{\Xinitv,0:T}{\chunk{Y}{0}{T}}$ as its unique invariant distribution.
The \pg sampler accomplishes this by making
use of \smc. From an algorithmic point of view, the difference between \pg and a standard \smc sampler is that in the former,
one particle trajectory, denoted as $x_{0:T}^\prime = \prange{x_0^\prime}{x_T^\prime} \in \Xset^{t+1}$, is specified \emph{a priori}.
This trajectory is used as a reference for the \pg sampler, as discussed below.

The reference trajectory is taken into account by simulating only $N-1$ particles in the usual way.
The $N$th particle is then set deterministically according to the reference. At the initialization, we thus simulate
independently $\{\epart{0}{i}\}_{i=1}^{N-1}$ with $\epart{0}{i} \sim \Xinitis{Y_0}{\cdot}$ and set $\epart{0}{N} = x_0^\prime$.
We then compute importance weights for all particles, $i = \range{1}{N}$, according to~\eqref{eq:particle-weight-0}.

Analogously, at any consecutive time point $t$, we sample the first $N-1$ particles
$\{ (A_t^i, \epart{t}{i}) \}_{i=1}^{N-1}$ conditionally independently given $\mcff{t-1}{N}$ according to \eqref{eq:ancestor-t}--\eqref{eq:particle-t}.
Note that these particles will depend on the reference trajectory through the resampling step \eqref{eq:ancestor-t}.
The $N$th particle and its ancestor index are then set deterministically: $\epart{t}{N} = x_t^\prime$ and $A_t^N = N$.
Finally, importance weights are then computed for all the particles according to \eqref{eq:particle-weight-t}.
Note that, by construction, the $N$th particle trajectory will coincide with the reference trajectory for all $t$, $\epart{0:t}{N} = x_{0:t}^\prime$.

After a complete pass of the above procedure, a trajectory $\epart{0:T}{\star}$ is sampled from among the
particle trajectories at time $T$ (see \eqref{eq:ancestral-path}), with probability proportional to the importance
weight $\ewght{T}{i}$, $i \in \{1,\dots, N\}$, \ie\,
\begin{equation}
\label{eq:probability-selection}
\CPP{\epart{0:T}{\star} = \epart{0:T}{i}}{ \mcff{T}{N} } = \frac{\ewght{T}{i}}{\sum_{\ell=1}^N \ewght{T}{\ell}} \eqsp, \quad i \in \{1,\dots,N\} \eqsp.
\end{equation}
This procedure thus associates each trajectory $x_{0:T}^\prime \in \Xset^{T+1}$ to a probability distribution on $(\Xset^{T+1}, \Xsigma^{\otimes(T+1)})$,
defining a Markov kernel on $(\Xset^{T+1}, \Xsigma^{\otimes(T+1)})$. More specifically, this kernel is given by
\begin{align}
  \label{eq:pg-kernel}
  \kernelPG(x_{0:T}^\prime, D) \eqdef \PE{ \frac{\sum_{i=1}^N \ewght{T}{i} \1_D(\epart{0:T}{i})}{\sum_{i=1}^N \ewght{T}{i}}},
\end{align}
for $(x_{0:T}^\prime, D) \in \Xset^{T+1} \times \Xsigma^{\otimes(T+1)}$,
where $\PE{}$ refers to expectation \wrt\ the random variables generated by the \pg algorithm.
We refer to $\kernelPG$ as the \pg kernel.

As shown by \citet{AndrieuDH:2010}, the conditioning on a reference trajectory implies that the \pg kernel leaves the \jsd invariant:
\begin{align*}
  \post{\Xinitv,0:T}{\chunk{Y}{0}{T}}(D) &= \int \kernelPG(x_{0:T}^\prime, D) \post{\Xinitv,0:T}{\chunk{Y}{0}{T}}(\rmd x_{0:T}^\prime)\,, & \forall D &\in \Xsigma^{(T+1)}.
\end{align*}
Quite remarkably, this invariance property holds for any $N\geq 1$.

Empirically, it has been found that the mixing of the \pg kernel can be improved significantly by
updating the ancestor indices $A_t^N$ for $t \in \{1,\dots,T\}$, either as part of the forward recursion \citep{LindstenJS:2012}
or in a separate backward recursion \citep{WhiteleyAD:2010}. We shall not specifically analyze these modified \pg algorithms
in this work, although our uniform ergodicity result apply straightforwardly to these algorithms as well.

\section{Main result}
\label{sec:main-result}
In this section we state the main results. First, in \autoref{sec:minorization}, we give a minorization condition for the \pg kernel.
Following this we discuss how to increase the number of particles $N=N_T$ as a function of the number of observations $T$ in order to
obtain a non-degenerate lower-bound. We consider first
a strong mixing condition and then a much weaker moment assumption in \autoref{sec:strong-mixing}~and~\autoref{sec:moment-assumption}, respectively.

\subsection{Minorization condition}\label{sec:minorization}
Define the sequence of nonnegative random variables $\{B_{t,T}\}_{t=0}^T$ by
\begin{equation}
\label{eq:def-b-T}
B_{t,T} = \sup_{0\leq \ell \leq T-t}
\frac{\supnorm{\ewghtfuncf{Y_t}} \supnorm{\Kunf{\chunk{Y}{t+1}{t+\ell}} \bigone}}{\cdens[\mu]{\chunk{Y}{t}{t+\ell}}{\chunk{Y}{0}{t-1}}}\\
\end{equation}
where, by convention, $\supnorm{\Kunf{\chunk{Y}{t+1}{t}} \bigone} = 1$.
\begin{theorem}
\label{theo:doeblin-condition-PG}
For all $x_{0:T}^\prime \in \Xset^{T+1}$ and $D \in \Xsigma^{\otimes (T+1)}$,
\begin{equation}
\label{eq:doeblin-condition-PG}
\kernelPG(x_{0:T}^\prime, D) \geq \mineps{T}{N} \eqsp \post{\Xinitv,0:T}{\chunk{Y}{0}{T}}(D) \eqsp,
\end{equation}
where
\begin{equation}
\label{eq:def-epsilon}
\mineps{T}{N} =   \prod_{t=0}^T \frac{N-1}{2 B_{t,T}+N-2} \eqsp.
\end{equation}
\end{theorem}
\begin{proof}
The proof is postponed to \autoref{sec:proof}.
However, to provide some intuition for the result, 
the main ideas of the proof are outlined below.

Using the representation of the \pg kernel from \eqref{eq:pg-kernel} we can write
\begin{align*}
  \kernelPG(x_{0:T}^\prime, D) \geq (N-1) \PE{\frac{ \ewght{T}{1} \1_D(\epart{0:T}{1})}{\sum_{i=1}^N \ewght{T}{i}}}
  \geq (N-1) \PE{\CPE{ \frac{ \ewght{T}{1} \1_D(\epart{0:T}{1})}{2\supnorm{\ewghtfuncf{Y_t}} + \sum_{i=2}^{N-1} \ewght{T}{i}} }{ \mcff{T-1}{N} }}
\eqsp,
\end{align*}
where, for the first inequality, we have simply discarded the $N$th term (corresponding to the reference particle) and used the fact that the $N-1$ weighted particles
$\{(\epart{0:T}{i}, \ewght{T}{i})\}_{i=1}^{N-1}$ are equally distributed. For the second inequality, we bound the first and the last term of the
sum in the denominator by $\supnorm{\ewghtfuncf{Y_t}}$. This has the effect that the random variables entering the numerator and the denominator of the expression
are conditionally independent given $\mcff{T-1}{N}$. By convexity of $x \mapsto 1/x$ and using Jensen's inequality we therefore obtain the bound
\begin{align*}
  \kernelPG(x_{0:T}^\prime, D) \geq  (N-1)  \PE{\frac{ \CPE{ \ewght{T}{1} \1_D(\epart{0:T}{1}) }{ \mcff{T-1}{N}} }
    { 2\supnorm{\ewghtfuncf{Y_t}}  + (N-2) \CPE{ \ewght{T}{2} }{\mcff{T-1}{N}} }}\eqsp.
\end{align*}
The inner conditional expectations can be computed explicitly. Principally, the result follows by repeating this procedure for time $T-1$, then for $T-2$, etc.
\end{proof}

\begin{corollary}\label{cor:finite-T}
  Assume that $g(x,y)>0$ for all $(x,y) \in \Xset \times \Yset$ and $\supnorm{\ewghtfuncf{y}}<\infty$ for all $y \in \Yset$. Then,
  for fixed $T$,
  $$
  \mineps{T}{N} \geq 1 + \frac{1}{N-1} \sum_{t=0}^{T} (1-2 B_{t,T}) + O_{\PP}(N^{-2}) \eqsp,
  $$
  and $\lim_{N\to \infty} \mineps{T}{N} = 1 $.
\end{corollary}
\begin{proof}
  From the definition \eqref{eq:def-epsilon} we have
  \begin{align*}
    \mineps{T}{N} = 
    \exp\left\{ - \sum_{t=0}^T \log\left( 1 + \frac{2 B_{t,T}-1}{N-1} \right) \right\}
    \geq \exp\left\{ \frac{1}{N-1} \sum_{t=0}^T \left( 1-2 B_{t,T} \right) \right\}\eqsp.
  \end{align*}
  For a fixed $T$,
  we thus obtain the result
  provided that $B_{t,T}<\infty$ for all $t \in \{0,\ldots,T\}$.
  However, the positivity of $g$ implies that $\cdens[\Xinitv]{\chunk{Y}{t}{t+\ell}}{\chunk{Y}{0}{t-1}}>0$ for all $\ell \geq 0$, and since
  $\supnorm{\ewghtfuncf{y}}<\infty$ for all $y \in \Yset$, it can be easily checked that
  $$
  \supnorm{\Kunf{\chunk{Y}{t+1}{t+\ell}} \bigone} \leq \prod_{s=t+1}^{t+\ell}  \supnorm{\ewghtfuncf{Y_s}}<\infty\eqsp,
  $$
  which immediately implies that $B_{t,T}<\infty$ for all $t \in \{0,\ldots,T\}$.
\end{proof}

\begin{remark}\label{rem:comparison-alv}
  The minorization condition of \autoref{theo:doeblin-condition-PG} is similar to Proposition~6 by \citet{AndrieuLV:2013}.
  However, they express the minorizing constant in terms of the expectation of a likelihood estimator \wrt\ the law of
  a ``doubly conditional SMC'' algorithm. They do not pursue an analysis of the effect on the minorization condition
  by the forgetting of the initial condition of the state space model. To obtain an explicit rate of convergence they
  assume, in our notation, that the triangular array of random variables $\{ B_{t,T} \}_{0\leq t \leq T}$ is uniformly
  bounded for $T \geq 0$. This is the case, basically, only when the model satisfies strong mixing conditions, as we discuss in
  the subsequent section. Indeed, \citet[Proposition~14~and~Lemma~17]{AndrieuLV:2013} is the same as our \autoref{prop:strong-mixing-2}.
\end{remark}


\subsection{Strong mixing condition}\label{sec:strong-mixing}
We first assume a strong mixing condition for the kernel $M$:
\begin{hyp}{S}
\item \label{assum:strong-mixing} There exist positive constants $(\sigma_-,\sigma_+)$, a nonnegative measure $\gamma$ and an integer $m \in \nset$ such that for all $x \in \Xset$,
$$
\sigma_- \gamma(\rmd x') \leq M^m(x,\rmd x') \leq \sigma_+ \gamma(\rmd x')\eqsp.
$$
\end{hyp}
This condition has been introduced by \citet{DelMoralG:1999} to establish the uniform-in-time convergence of the particle filter. This condition, which is stronger than the Doeblin condition, typically requires that the state space is compact. It is overly restrictive but is often use in the analysis of state space models because it implies a form of uniform forgetting of the initial condition of the filter, which is key to obtaining long-term stability of the particle filter.
\begin{proposition}
\label{prop:strong-mixing}
Assume that \ref{assum:strong-mixing} holds with $m=1$ and that the proposal kernel is fully-adapted as defined in \eqref{eq:bootstrap-fully-adapted}. Then, taking $N_T \sim \lambda T$ for some $\lambda>0$, we have
$$
\liminf_{T \to \infty} \mineps{T}{N_T} \geq \exp\left(\frac{1-2 (\sigma_+/\sigma_-)^2}{\lambda}\right)>0\,, \quad \as
$$
\end{proposition}
\begin{proof}
First, note that for all $\ell \geq 1$,
\begin{equation}
\label{eq:majo:unif:mix}
\supnorm{\Kunf{\chunk{Y}{t+1}{t+\ell}} \bigone} \leq \sigma_+ \int \gamma(\rmd x_{t+1}) g(x_{t+1},Y_{t+1}) \Kunf{\chunk{Y}{t+2}{t+\ell}} \bigone (x_{t+1})
\end{equation}
and
\begin{align}
  \label{eq:mino:unif:mix}
  \cdens[\Xinitv]{\chunk{Y}{t}{t+\ell}}{\chunk{Y}{0}{t-1}}\geq \sigma_-^2 \int \gamma(\rmd x_{t})g(x_{t},Y_{t})
  \int \gamma(\rmd x_{t+1}) g(x_{t+1},Y_{t+1}) \Kunf{\chunk{Y}{t+2}{t+\ell}} \bigone (x_{t+1})\eqsp.
\end{align}
Now, in the fully-adapted case, we have:
$$
\Kis{y}{x}{\rmd x'}=\frac{M(x,\rmd x')g(x',y)}{\int M(x,\rmd u)g(u,y)}\eqsp,
$$
so that by the definition of $\ewghtfuncf{y}$,
$$
\supnorm{\ewghtfuncf{Y_t}}=\sup_{x \in \Xset} \left|\int M(x,\rmd x_t)g(x_,Y_t)\right| \leq \sigma_+ \int \gamma(\rmd x_t) g(x_t,Y_t)\eqsp.
$$
Combining this equality with \eqref{eq:majo:unif:mix} and \eqref{eq:mino:unif:mix} yields:
$$
B_{t,T} = \sup_{0\leq \ell \leq T-t}
\frac{\supnorm{\ewghtfuncf{Y_t}} \supnorm{\Kunf{\chunk{Y}{t+1}{t+\ell}} \bigone}}{\cdens[\Xinitv]{\chunk{Y}{t}{t+\ell}}{\chunk{Y}{0}{t-1}}} \leq \left(\frac{\sigma_+}{\sigma_-}\right)^2 \eqsp.
$$
By the definition \eqref{eq:def-epsilon}, we then obtain:
$$
\mineps{T}{N} \geq  \prod_{t=0}^T \frac{N-1}{2 B_{t,T}+N-2} \geq \left( \frac{N-1}{N-2+2 (\sigma_+/\sigma_-)^2} \right)^{T+1} \eqsp.
$$
Finally, letting $N_T \sim \lambda T$, we obtain
$$
\liminf_{T \to \infty} \mineps{T}{N_T} \geq \exp\left(\frac{1-2 (\sigma_+/\sigma_-)^2}{\lambda}\right)>0\,, \quad \as
$$
\end{proof}
It is worthwhile to stress that \autoref{prop:strong-mixing} holds whatever the distribution of the observation process, $\sequence{Y}[t][\nset]$ is. This is a consequence of the strong mixing condition \ref{assum:strong-mixing} which provides a simple result, but at the expense of an assumption which is rarely met in practice. If instead of the fully adapted case, we consider the bootstrap filter (see \eqref{eq:bootstrap-fully-adapted}), we may also obtain a uniform-in-time bound. However, this requires an even stronger assumption of the existence of a lower and an upper bound for the observation likelihood.
\begin{hyp}{S}
\item \label{assum:strong-mixing-likelihood} There exists a positive constant $\delta$, such that for  all $y \in \Yset$,
$$
1 \leq \frac{\sup_{x \in \Xset} g(x,y)}{\inf_{x \in \Xset} g(x,y)} \leq \delta \eqsp.
$$
\end{hyp}

\begin{proposition}
\label{prop:strong-mixing-2}
Assume that \ref{assum:strong-mixing}-\ref{assum:strong-mixing-likelihood} hold and that the bootstrap proposal is used: $\Kis{y}{x}{\cdot}= M(x,\cdot)$.
Then, taking $N_T \sim \lambda T$ for some $\lambda>0$, we have
$$
\liminf_{T \to \infty} \mineps{T}{N_T} \geq \exp\left(\frac{1-2 \delta^m \sigma_+/\sigma_- }{\lambda}\right)>0\,, \quad \as \eqsp,
$$
where $m$ is defined in \ref{assum:strong-mixing}.
\end{proposition}
\begin{proof}
For the bootstrap filter, $\ewghtfunc{y}{x}{x'}= g(x',y)$. Therefore, $\supnorm{\ewghtfuncf{y}}= \sup_{x \in \Xset} g(x,y)$. On the other hand, for $\ell \geq m$,
\begin{align}
\label{eq:majo:unif:mix-1}
\supnorm{\Kunf{\chunk{Y}{t+1}{t+\ell}} \bigone} \leq \sigma_+ \left( \prod_{s=t+1}^{t+m-1} \sup_{x \in \Xset} g(x,Y_s) \right)
\int \gamma(\rmd x_{t+m}) g(x_{t+m},Y_{t+m}) \Kunf{\chunk{Y}{t+m+1}{l}} \bigone(x_{t+m}) \eqsp,
\end{align}
and
\begin{align}
\label{eq:mino:unif:mix-1}
\cdens[\Xinitv]{\chunk{Y}{t}{t+\ell}}{\chunk{Y}{0}{t-1}}\geq  \sigma_-  \left( \prod_{s=t}^{t+m-1} \inf_{x \in \Xset} g(x,Y_s) \right)
 \int \gamma(\rmd x_{t+m}) g(x_{t+m},Y_{t+m}) \Kunf{\chunk{Y}{t+m+1}{t+\ell}} \bigone(x_{t+m}) \eqsp.
\end{align}
Combining \eqref{eq:majo:unif:mix-1} and \eqref{eq:mino:unif:mix-1} yields
\begin{equation}
B_{t,T} \leq \delta^m \frac{\sigma_+}{\sigma_-} \eqsp.
\end{equation}
The result follows as in the proof of \autoref{prop:strong-mixing}.
\end{proof}

\subsection{Moment assumption}\label{sec:moment-assumption}
Under the strong mixing condition \ref{assum:strong-mixing} and the even more restrictive \ref{assum:strong-mixing-likelihood},
we obtained non-degenerate uniform convergence bounds when the number of trajectories $N_T$ depends linearly on the number of observations $T$.
However, these conditions are very restrictive and hardly ever satisfied when the state space is non-compact.
We now turn to the analysis of the minorization condition under a much weaker moment assumption.
However, when the strong mixing assumption is relaxed, we are no longer able to obtain bounds that hold uniformly \wrt\ the observation sequence.
Instead, we will take a probabilistic approach. In \autoref{prop:N-depend-on-T} below, we
show that the minorizing constant will be bounded away from zero, in probability,
provided that $N_T$ is a power of $T$.
Moreover, the result presented in this section is not restricted to the fully-adapted or the bootstrap \pg kernel and may be obtained for virtually any proposal kernel. 

This result holds \wrt\ the law of the observation process $\sequence{Y}[t][\nset]$. It is therefore of interest
to carry out the analysis for a parametric family of state space models $\{ (M^\param, G^\param), \param \in \Param\}$,
where $\Param$ is a compact subset of a Euclidean space.
Informally, this allows us to analyse the ergodicity of the \pg kernel, even when the algorithm is executed using a misspecified model.
We consider a sequence of parameters $\sequence{\param}[T][\nset]$ that become increasingly close to some ``true'' parameter $\param_\star$
(in a sense that will be made precise in \autoref{prop:N-depend-on-T} below), converging at a rate $1/\sqrt{T}$.
The rationale for this assumption is that we are considering the large $T$ regime and we can therefore expect $\param_T$ to be close to $\tparam$. We discuss this further in \autoref{rem:kullback} below.
Note that, for a fixed observation sequence $\chunk{Y}{0}{T}$ (with finite $T$) we can instead appeal to \autoref{cor:finite-T}.

\begin{hyp}{A}
\item \label{assum:stationarity} For all $\param \in \Param$, the kernel $M^\param$ has a unique stationary distribution denoted as $\pi^\param$.
\end{hyp}
In what follows, for $\param\in\Param$
we let $\PE[\param]{}[\mu]$ and $\PP^\param_\mu$ refer to the expectation and probability, respectively, induced on $((\Xset\times\Yset)^{\nset},(\Xsigma\otimes\Ysigma)^{\otimes \nset})$ by a
Markov chain $\{(X_t, Y_t), t\in\nset\}$ evolving according to the state space model $(M^\param, G^\param)$ starting with $X_0\sim \mu$.
For simplicity, we write $\PEstat[\param]{}=\PE[\param]{}[\pi^\param]$ and $\PPstat^\param=\PP^\param_{\pi^\param}$.
For $1 \leq s \leq t$, we write
$$  \dens[\mu,s]{\chunk{y}{s}{t}}[\param] = \int \dens[\mu]{\chunk{y}{0}{t}}[\param]\kappa^{\otimes s}(\rmd y_{0:s-1}) $$
for the marginal probability density function of $Y_{s:t}$ \wrt\ $\kappa^{\otimes(t-s+1)}$.

Define for all $(t,\ell) \in \nset \times \nset^\star$,
\begin{align}
\label{eq:def-bar-b-t-ell}
&\Bbar[t]{t+\ell}{\mu}{\param} \eqdef \frac{\supnorm{\ewghtfuncf[\param]{Y_t}} \supnorm{\Kunf[\param]{\chunk{Y}{t+1}{t+\ell}} \bigone}}{\dens[\mu,t]{\chunk{Y}{t}{t+\ell}}[\param]}\eqsp, \\
\label{eq:def-bar-c-t-ell}
&\Cbar[t]{t+\ell}{\mu}{\param} \eqdef
\frac{\supnorm{\ewghtfuncf[\param]{Y_t}} \int \lambda(\rmd x_{t+1}) g^\param(x_{t+1},Y_{t+1}) \Kunf[\param]{\chunk{Y}{t+2}{t+\ell}} \bigone(x_{t+1})}{\dens[\mu,t]{\chunk{Y}{t}{t+\ell}}[\param]} \eqsp,
\end{align}
with, by convention,
\begin{align*}
\Bbar{t}{\mu}{\param} &= \supnorm{\ewghtfuncf[\param]{Y_t}}/ \dens[\mu,t]{Y_t}[\param] \eqsp, &
\Cbar[t]{t+1}{\mu}{\param} &= \supnorm{\ewghtfuncf[\param]{Y_t}} \int \lambda(\rmd x_{t+1}) g^\param(x_{t+1},Y_{t+1}) / \dens[\mu,t]{\chunk{Y}{t}{t+1}}[\param] \eqsp .
\end{align*}

\begin{hyp}{A}
\item \label{assum:m-upper-bound} There exists a constant $\sigma_+ \in \rset^+$ and a nonnegative measure $\lambda$ such that for all $\param \in\Param$ and $(x,A)\in \Xset \times \Xsigma$,
$$
M^\param(x,A) \leq \sigma_+ \lambda(A)\eqsp.
$$
\end{hyp}
Denote by $m^\param(x,\cdot)$ the Radon-Nikodym derivative
\begin{equation}
\label{eq:definition-density}
m^\param(x,x')=\frac{\rmd M^\param(x,\cdot)}{\rmd \lambda(\cdot)}(x')\eqsp.
\end{equation}
Under \ref{assum:m-upper-bound}, the stationary distribution $\pi^\param$ is absolutely continuous \wrt\ $\lambda$.
Furthermore, for notational simplicity
it is assumed that 
the initial distribution $\mu$ is absolutely continuous \wrt\ $\lambda$.
By abuse of notation, we write  $\pi^\param$ and $\mu$ also for the corresponding density functions.
\begin{hyp}{A}
\item \label{assum:m-g-positive} For all $\param \in\Param$ and  $(x,x',y) \in \Xset^2 \times \Yset$, $m^\param(x,x')>0$ and $ g^\param(x,y)>0$.
\end{hyp}

\begin{hyp}{A}
\item \label{assum:b-moment-bound}
There exist constants $(\ell_\star,\alpha) \in \nset^\star \times (0,1)$ such that,
\begin{align}
& \sup_{t \in \nset}\sup_{\param \in \Param} \PE[\param]{( \Bbar[t]{t+\ell}{\mu}{\param} )^\alpha}[\mu]    <\infty\,, \quad \mbox{for all } \ell \in \{0,\ldots, \ell_\star-1\} \eqsp, \label{eq:bound-moment-B}\\
& \sup_{t \in \nset} \sup_{\param \in \Param} \PE[\param]{( \Cbar[t]{t+\ell_\star}{\mu}{\param} )^\alpha}[\mu]   <\infty\eqsp. \label{eq:bound-moment-tilde-B}
\end{align}
\end{hyp}

\begin{theorem} \label{prop:N-depend-on-T}
Assume that \ref{assum:stationarity}, \ref{assum:m-upper-bound}, \ref{assum:m-g-positive}, and \ref{assum:b-moment-bound} hold.
Let $\tparam \in \Param$ and 
let $\sequence{\param}[T][\nset]$ be a sequence of parameters such that
\begin{equation}
\label{eq:nuit-de-singapour}
\limsup_{T \to \infty}  T
\PEstat[\tparam]{\ln\left(\frac{ m^{\tparam}(X_0,X_1)g^{\tparam}(X_1,Y_1)}
{ m^{\param_T}(X_0,X_1)g^{\param_T}(X_1,Y_1)}\right)}
 < \infty
\eqsp.
\end{equation}
Furthermore, assume that
\begin{equation}
  \label{eq:nuit-de-singapour-2}
  \PEstat[\tparam]{\ln\left(\frac{ \pi^{\tparam}(X_0)}{ \mu(X_0)}\right)} < \infty
   \eqsp,
\end{equation}
where $\mu$ is the initial distribution used in the \pg algorithm.
Then, for all $0\leq \gamma<\alpha$ (where $\alpha$ is defined in \ref{assum:b-moment-bound})
and for all sequences of integers $\{N_T\}_{T \geq 1}$ such that $N_T \sim T^{1/\gamma}$,
the sequence $\{\mineps{T}{N_T}^{-1}(\param_T)\}_{T \geq 1}$, defined in \eqref{eq:def-epsilon}  is $\PPstat^\tparam$-tight (bounded in probability).
\end{theorem}
\begin{proof}
The proof is postponed to \autoref{sec:proof:prop:N-depend-on-T}.
\end{proof}

\begin{remark} \label{rem:kullback}
For any $\param \in \Param$,
  \begin{align*}
  \operatorname{D}(\tparam || \param) \eqdef \PEstat[\tparam]{\ln\left(\frac{ m^{\tparam}(X_0,X_1)g^{\tparam}(X_1,Y_1)}
        { m^{\param}(X_0,X_1)g^{\param}(X_1,Y_1)}\right)} \eqsp,
  \end{align*}
  is the expectation under the stationary distribution $\tpi$ of the Kullback-Leibler divergence between
  the conditional distribution of $\cdens{X_1,Y_1}{X_0}[\tparam]$ and $\cdens{X_1,Y_1}{X_0}[\param]$. Hence,
  $\operatorname{D}(\tparam || \param) \geq 0$ for all $\param \in \Param$ and  $\operatorname{D}(\tparam || \tparam)=0$.
  Assuming that $\tparam$ belongs to the interior of $\Param$ and that the function $\param \mapsto \operatorname{D}(\tparam||\param)$ is twice differentiable at $\tparam$, a Taylor expansion at $\tparam$ yields
  \begin{align*}
  \operatorname{D}(\tparam||\param)=  \frac{1}{2} (\tparam - \param)^t H^\tparam (\tparam - \param) + o( \Vert \tparam - \param \Vert^2) \eqsp,
  \end{align*}
  where $H^\param$ is the Hessian of $\param \mapsto \operatorname{D}(\tparam|| \param)$. 
  Consequently, for regular statistical models, \eqref{eq:nuit-de-singapour} holds provided that $\theta_T$ converges to $\tparam$ at a rate $1/\sqrt{T}$,
  \ie,
  $$
  \param_T = \tparam + \varrho_T/ \sqrt{T}\eqsp,
  $$
  where the sequence $\sequence{\varrho}[T][\nset]$ is bounded: $\sup_{T \geq 0} \| \varrho_T \| < \infty$.
\end{remark}

\begin{remark}\label{rem:deterministic}
  It should be noted that our results do not cover explicitly the case when the sequence $\sequence{\varrho}[T][\nset]$ is stochastic.
  Still, we believe that our results hint at the possibility of obtaining a non-degenerate lower bound on the
  minorizing constant also in the stochastic case, given that $\sequence{\varrho}[T][\nset]$ is tight, under conditions that are much weaker than the previously considered
  strong mixing assumption.
\end{remark}

\begin{remark}
  It is interesting to note that we do not require the initial distribution $\mu$ to be equal to $\tpi$,
  but only that the Kullback-Leibler divergence \eqref{eq:nuit-de-singapour-2} is bounded.
  Hence, we may use a quite arbitrary initial distribution and still obtain a sequence of inverse minorization constants that is tight \wrt\ $\PPstat^\tparam$.
\end{remark}
A straightforward generalization of the above result is to let the initial distribution belong to a parametric family of distributions,
$\{\mu^\param : \param\in\Param\}$. The condition \eqref{eq:nuit-de-singapour-2} should then be replaced by
\begin{align*}
  \limsup_{T\to\infty} \PEstat[\tparam]{\ln\left(\frac{ \pi^{\tparam}(X_0)}{ \mu^{\param_T}(X_0)}\right)}
  < \infty
  \eqsp.
\end{align*}
Allowing for the initial distribution to depend on $\param$ can be useful in some cases.  For instance, if the stationary distribution $\pi^\theta$ is known
it may serve as a natural choice for the initial distribution used in the algorithm.

\section{Examples}\label{sec:examples}%
In this section we consider two examples to illustrate how the assumptions of \autoref{prop:N-depend-on-T}
can be verified in practice.
We preface the examples by a technical lemma, which will be very useful for checking the assumptions.

\begin{lemma}
\label{lem:holder-inequality}
Let $(\Zset,\Zsigma)$ be a measurable set and $\xi$ be a measure on $(\Zset,\Zsigma)$.
Let $\alpha\in\ooint{0,1}$ and let $\varphi,\psi$ and $q$ be  nonnegative measurable functions, such that
\begin{align}
\label{eq:conditions-holder-1}
&\int \psi(z) \varphi(z) \xi(\rmd z) < \infty \eqsp,\\
\label{eq:conditions-holder-2}
&\int \varphi^{-\frac{\alpha}{1-\alpha}}(z) q(z) \xi(\rmd z) < \infty \eqsp.
\end{align}
Then,
\begin{equation}
\label{eq:hoder-inequality-conclusion}
\int \psi^\alpha(z) q^{1-\alpha}(z) \xi (\rmd z) < \infty \eqsp.
\end{equation}
\end{lemma}
\begin{proof}
The result follows from H{\"o}lder's inequality:
\begin{align*}
\int \psi^\alpha(z) q^{1-\alpha}(z) \xi(\rmd z)
&= \int \left[ \psi(z) \varphi(z) \right]^\alpha  \, \left[\varphi^{-\frac{\alpha}{1-\alpha}}(z) q(z)\right]^{1-\alpha} \xi(\rmd z) \\
&\leq \left( \int \psi(z) \varphi(z) \xi(\rmd z) \right)^{\alpha} \left( \int \varphi^{-\frac{\alpha}{1-\alpha}}(z) q(z)
\xi(\rmd z) \right)^{1-\alpha} \eqsp.
\end{align*}
\end{proof}

\subsection{A nonlinear model with additive measurement noise}\label{sec:examples:addtive-noise}
\newcommand\ANparam{\xi}
\newcommand\ANParam{\Xi}
We consider first a class of nonlinear state space models where the latent process is observed in additive noise,
\begin{align}
\label{eq:state-equation-additive-noise}
X_{t+1}&= h^\ANparam(X_t) + \sigma_W W_{t+1} \\
\label{eq:measuremant-equation-additive-noise}
Y_t &= \phi X_t + \sigma_U U_t
\end{align}
where $\sequence{W}[t][\nset]$ and $\sequence{U}[t][\nset]$ are two independent sequences of \iid\ standard Gaussian random variables and
$\{ h^\ANparam, \ANparam \in \ANParam \}$ is a parametric family of measurable real-valued functions, where $\ANParam$ is a compact subset of a Euclidean space. We denote by $\param = (\ANparam,\phi,\sigma_U,\sigma_W)$ the parameters of the model. It is assumed that $\param \in \Param$, where
$\Param$ is a compact subset of $\ANParam \times \ooint{0,\infty}^3$. We assume that for all $\ANparam \in \ANParam$, $x \mapsto h^\ANparam(x)$ is continuous and $\sup_{\ANparam \in \ANParam} \limsup_{x \to \infty} |h^\ANparam(x)|/|x| < 1$. For any $\delta > 0$, we set $V_\delta(x) = \rme^{\delta |x|}$. It is easily seen that there exist constants $\lambda_\delta \in \ooint{0,1}$ and $b_\delta < \infty$ such that
\begin{equation}
\label{eq:drift-condition-additive-noise}
\sup_{\param \in \Param} \PE[\param]{ V_\delta( X_1) }[x] \leq \lambda_\delta V_\delta(x) + b_\delta \eqsp.
\end{equation}
The Markov chain is strong Feller, Harris recurrent, all the compact sets are small, and the Markov chain admits a single invariant distribution. Therefore, \ref{assum:stationarity} and \ref{assum:m-upper-bound} are satisfied.
Since both the transition density and the observation density are Gaussian, \ref{assum:m-g-positive} is also readily satisfied.
We will thus focus on verifying the moment assumption \ref{assum:b-moment-bound}.

First, note that
\begin{equation}
\label{eq:moment-control-additive-noise}
\sup_{t \in \nset} \sup_{\param \in \Param} \PE[\param]{V_\delta(X_t)}[x]
\leq  \lambda_\delta^t V_\delta(x) + b_\delta(1+\lambda_\delta+\cdots+\lambda_\delta^{t-1})
\leq  V_\delta(x) + b_\delta / (1-\lambda_\delta) \eqsp.
\end{equation}

We assume that the initial distribution $\mu$ is such that $\mu(V_\delta) < \infty$. Therefore,
\begin{equation}
\label{eq:moment-control-additive-noise-1}
\sup_{t \in \nset} \sup_{\param \in \Param} \PE[\param]{V_\delta(X_t)}[\mu] < \infty \eqsp.
\end{equation}
Interestingly for the model \eqref{eq:state-equation-additive-noise}--\eqref{eq:measuremant-equation-additive-noise} it is possible
to use the fully adapted proposal kernel \citep{DoucetGA:2000} as defined in \eqref{eq:bootstrap-fully-adapted}, for which
\begin{align}
  \nonumber
  \ewghtfunc[\param]{y}{x}{x'}&= \int m^\param(x, x^{\prime\prime}) g^\param(x^{\prime\prime},y)\rmd x^{\prime\prime} \\
  \label{eq:additive-noise:definition-weightfunc}
  &= \frac{1}{\sqrt{2\pi (\phi^2\sigma_W^2 + \sigma_U^2) }}\exp\left( -\frac{1}{2(\phi^2\sigma_W^2 + \sigma_U^2)}\left(y-\phi h^\ANparam(x) \right)^2 \right) \eqsp,
\end{align}
for all $(x,x') \in \rset \times \rset$, $y \in \rset$, and $\param \in \Param$.
It can be seen that, for any $\param\in\Param$ and any $y \in \rset$,
\[
\int_{-\infty}^{\infty} g^\param(x,y) \rmd x = \frac{1}{\phi} \eqsp,   \quad \text{and} \quad
\supnorm{\ewghtfuncf[\param]{y}} \leq \frac{1}{\sqrt{2\pi (\phi^2\sigma_W^2 + \sigma_U^2) }} \eqsp,
\]
which implies the existence of constants $D_1$ and $D_2$ such that
\begin{equation}
\label{eq:borne-util-additive-noise}
\sup_{\param \in \Param} \int_{-\infty}^{\infty} g^\param(x,y) \rmd x  \leq D_1\eqsp,   \quad \text{and} \quad
\sup_{\param \in \Param} \supnorm{\ewghtfuncf[\param]{y}}  \leq D_2 \eqsp.
\end{equation}
Analogous bounds hold also if we would instead consider the bootstrap proposal (see \eqref{eq:bootstrap-fully-adapted}).

To verify \ref{assum:b-moment-bound} we let $\ell_\star=1$ and show that,
\begin{equation}
\label{eq:additive-noise:moment:A4}
\sup_{t \in \nset} \sup_{\param\in\Param} \PE[\param]{ ( \Bbar{t}{\mu}{\param})^\alpha}[\mu]<\infty\, , \quad
\sup_{t \in \nset} \sup_{\param\in\Param} \PE[\param]{ (\Cbar[t]{t+1}{\mu}{\param})^\alpha}[\mu] <\infty\eqsp,
\end{equation}
for some (and actually any) $\alpha \in \coint{0,1}$.
Consider first
\begin{equation}
\label{eq:additive-noise:borne-B00alpha}
\PE[\param]{ (\Bbar{t}{\mu}{\param} )^\alpha}[\mu] = \int \supnorm[\alpha]{\ewghtfuncf[\param]{y_t}} \{\dens[\mu,t]{y_t}[\param]\}^{1-\alpha} \rmd y_t
\leq D_2^\alpha \int \{\dens[\mu,t]{y_t}[\param]\}^{1-\alpha} \rmd y_t \eqsp,
\end{equation}
where the inequality follows from \eqref{eq:borne-util-additive-noise}.
We apply \autoref{lem:holder-inequality} to establish a bound for the right-hand side of \eqref{eq:additive-noise:borne-B00alpha}.
Let $\psi(y) = 1$ and $\varphi(y) = 1/(1\vee |y|^2)$.
With these definitions the first condition in \eqref{eq:conditions-holder-1} is satisfied. To check \eqref{eq:conditions-holder-2},
note that
\begin{equation*}
  \varphi^{-\frac{\alpha}{1-\alpha}}(y) = (1 \vee |y|^2)^{\frac{\alpha}{1-\alpha}} \leq 1+|y|^{2 \alpha/(1-\alpha)} \eqsp.
\end{equation*}
The integral in \eqref{eq:conditions-holder-2} may be expressed as
\begin{equation}
  \label{eq:additive-noise:what-we-need}
  \int \varphi^{-\frac{\alpha}{1-\alpha}}(y_t) \dens[\mu,t]{y_t}[\param] \rmd y_t = \PE[\param]{\varphi^{-\frac{\alpha}{1-\alpha}}(Y_t)}[\mu]=
  \PE[\param]{\PE[\param]{\varphi^{-\frac{\alpha}{1-\alpha}}(Y_0)}[X_t]}[\mu] \eqsp.
\end{equation}
Since $Y_0= \phi X_0 + \sigma U_0$, we get that for any $x \in \Xset$,
$\PE[\param]{ \varphi^{-\frac{\alpha}{1-\alpha}}(Y_0) }[x] \leq 1 + \PE{ |\phi x + U|^{2\alpha/(1-\alpha)}}$,
where $U$ is standard normal. This implies that there exists a constant $D_3$ such that, for all $x \in \Xset$ and all $\param\in\Param$,
\begin{equation}
\label{eq:gainesville}
\PE[\param]{ \varphi^{-\frac{\alpha}{1-\alpha}}(Y_0) }[x]
\leq D_3 (1+ |x|^{2\alpha/(1-\alpha)}) \eqsp.
\end{equation}
Plugging this into \eqref{eq:additive-noise:what-we-need} and using \eqref{eq:moment-control-additive-noise-1}, this verifies the second condition in \eqref{eq:conditions-holder-2}.
\autoref{lem:holder-inequality} can thus be used to conclude that
$ \PE[\param]{ (\Bbar{t}{\mu}{\param})^\alpha}[\mu] < \infty$ for all $\alpha \in \ooint{0,1}$.
Since this holds for any $t \in \nset$ and $\param\in\Param$, we obtain the first part of \eqref{eq:additive-noise:moment:A4}.

Next, we consider
\begin{align}
\nonumber
\PE[\param]{(\Cbar[t]{t+1}{\mu}{\param})^\alpha}[\mu]
&= \PE[\param]{\frac{\supnorm[\alpha]{\ewghtfuncf[\param]{Y_t}} \left( \int g^\param(x_{t+1},Y_{t+1}) \rmd x_{t+1} \right)^\alpha}{ \{\dens[\mu,t]{\chunk{Y}{t}{t+1}}[\param] \}^\alpha}}[\mu]\\
\nonumber
&\leq D_1^\alpha D_2^\alpha  \iint  \{ \dens[\mu,t]{\chunk{y}{t}{t+1}}[\param] \}^{1-\alpha} \rmd \chunk{y}{t}{t+1} \eqsp.
\label{eq:additive-noise:borne-B01alpha}
\end{align}
We will again make use of \autoref{lem:holder-inequality} to bound this quantity.
Proceeding analogously to above, we let $\psi(y_0,y_1) = 1$ and
\begin{equation}
\varphi(y_0,y_1) = \frac{1}{(y_0^2 \vee 1) \, (y_1^2 \vee 1)} \eqsp,
\label{eq:additive-noise:definition-phi-1}
\end{equation}
for which \eqref{eq:conditions-holder-1} is satisfied.
To check \eqref{eq:conditions-holder-2}, we use the conditional independence of the observations given the states and \eqref{eq:additive-noise:what-we-need} to get, for any $\param \in \Param$,
\begin{multline}
\nonumber
\iint \varphi^{-\frac{\alpha}{1-\alpha}}(\chunk{y}{t}{t+1}) \dens[\mu,t]{\chunk{y}{t}{t+1}}[\param] \rmd \chunk{y}{t}{t+1}
= \PE[\param]{\CPEdoup[\mu]{\param}{\varphi^{-\frac{\alpha}{1-\alpha}}(\chunk{Y}{t}{t+1})}{ X_{t:t+1} }}[\mu] \\
\leq \PE[\param]{\PE[\param]{1 + |Y_0|^{\frac{2\alpha}{1-\alpha}}}[X_t] \PE[\param]{1 + |Y_0|^{\frac{2\alpha}{1-\alpha}}}[X_{t+1}] }[\mu] \leq D_3^2 \PE[\param]{(1+ |X_t|^{\frac{2\alpha}{1-\alpha}}) (1+ |X_{t+1}|^{\frac{2\alpha}{1-\alpha}})}[\mu]
\eqsp.
\label{eq:additive-noise:interediary-result-1}
\end{multline}
From the Cauchy-Schwarz inequality we get, by using \eqref{eq:moment-control-additive-noise-1},
\begin{equation*}
\sup_{t \in \nset}
\sup_{\param \in \Param} \PE[\param]{\varphi^{-\frac{\alpha}{1-\alpha}}(\chunk{Y}{t}{t+1})}[\mu]
\leq D_3^2 \sup_{t \in \nset} \sup_{\param \in \Param} \PE[\param]{(1+|X_t|^{\frac{2\alpha}{1-\alpha}})^2}[\mu] < \infty \eqsp.
\end{equation*}
This shows that \eqref{eq:conditions-holder-2} is satisfied for any $\param\in\Param$ and any $t\in\nset$ which, by \autoref{lem:holder-inequality}
implies $\sup_{t\in\nset}\sup_{\param\in\Param}\PE[\param]{(\Cbar[t]{t+1}{\mu}{\param})^\alpha} < \infty$  for all $\alpha \in \coint{0,1}$, verifying
\ref{assum:b-moment-bound}.

Provided that $\theta_T$ converges to $\tparam$ at a rate $1/\sqrt{T}$ (see \autoref{rem:kullback}),
we may therefore apply \autoref{prop:N-depend-on-T} which shows that for any $\gamma \in \ooint{0,1}$, $\{\mineps{T}{N_T}^{-1}(\param_T) \}_{T \geq 1}$ is tight with $N_T \sim T^{1/\gamma}$.

\subsection{A stochastic volatility model}

The canonical model in stochastic volatility for discrete-time data has been introduced by
\cite{taylor:1982}  and worked out since then by many authors; see  \cite{hull:white:1987} and \cite{jacquier:polson:rossi:1994} for early references and \cite{shephard:andersen;2009} for an up-to-date survey. In this model, the hidden
volatility process, $\sequence{X}[t][\nset]$,
follows a first order autoregression,
\begin{align}
\label{eq:stochasticvolatilitycanonical-hidden}
X_{t+1} & = \phi X_t + \sigma W_{t+1} \,, \\
\label{eq:stochasticvolatilitycanonical-observation}
Y_t & = \beta \exp(X_t/2) U_t  \,.
\end{align}
where $\sequence{W}[t][\nset]$ and
$\sequence{U}[t][\nset]$ are white Gaussian noise with mean zero and unit variance. The error processes
$\sequence{W}[t][\nset]$ and $\sequence{U}[t][\nset]$
are assumed to be mutually independent. We denote by $\param= (\phi,\sigma,\beta) \in \Param$, where $\Param$ is
a compact subset of $\ooint{-1,1} \times \ooint{0,\infty}^2$.
For $\delta > 0$, denote by $V_\delta(x)= \rme^{\delta |x|}$ and let
$\mu$ an arbitrary distribution on $(\rset, \borel(\rset))$, for which $\mu(V_\delta) < \infty$.

For this model the transition kernel and the likelihood of the observation are given by
\begin{align}
\label{eq:likelihood-SV}
&m^\param(x,x')= \frac{1}{\sqrt{2 \pi \sigma^2}} \exp\left( - \frac{1}{2 \sigma^2} (x'- \phi x)^2 \right) \qquad\text{and} \\
&g^\param(x,y)= \frac{1}{\sqrt{2 \pi \beta^2}} \rme^{-(x/2 + (y^{2}/2\beta^2) \rme^{-x})} \eqsp,
\end{align}
respectively. For any $\param \in \Param$, the autoregressive process $\sequence{X}[t][\nset]$ has a unique stationary distribution $\pi^\param$, which is Gaussian, with mean 0 and variance $\sigma^2/(1-\phi^2)$. Hence, \ref{assum:stationarity} is satisfied.

We consider the bootstrap proposal kernel as defined in \eqref{eq:bootstrap-fully-adapted}, in which case
\begin{equation}
\label{eq:definition-weightfunc-SV}
\ewghtfunc[\param]{y}{x}{x'}= g^\param(x',y) \eqsp, \quad \text{for all $(x,x') \in \rset \times \rset$ and $y \in \rset$} \eqsp.
\end{equation}
Note that, $\supnorm{\ewghtfuncf[\param]{y}}= \supnorm{g^\param(\cdot,y)}$.
Assumptions \ref{assum:m-upper-bound} and \ref{assum:m-g-positive} are readily satisfied.
We finally check \ref{assum:b-moment-bound}. It is easily shown that, for all $\param \in \Param$,
\begin{align}
\label{eq:borne-util-SV-1}
&\int_{-\infty}^{\infty} g^\param(x,y) \rmd x = \frac{D_1}{|y|} \eqsp,  && D_1 \eqdef \frac{1}{\sqrt{2 \pi}} \int_0^\infty \frac{\rme^{-u/2}}{\sqrt{u}} \rmd u \eqsp \\
\label{eq:borne-util-SV-2}
&\sup_{x \in \rset} g^\param(x,y) = \frac{D_2}{|y|} \eqsp, && D_2 \eqdef \frac{1}{\sqrt{2 \pi \rme}} \eqsp.
\end{align}
Similarly to \autoref{sec:examples:addtive-noise} we will check \ref{assum:b-moment-bound} with $\ell_\star=1$, \ie, we show that
\begin{equation}
\label{eq:moment:A4}
\sup_{t \in \nset} \, \sup_{\param\in\Param} \PE[\param]{ ( \Bbar{t}{\mu}{\param} )^\alpha}[\mu]<\infty\, , \quad
\sup_{t \in \nset} \, \sup_{\param\in\Param} \PE[\param]{ (\Cbar[t]{t+1}{\mu}{\param})^\alpha}[\mu]<\infty\eqsp,
\end{equation}
for any $\alpha \in (0,1)$. Note that we cannot expect \eqref{eq:moment:A4} to hold with $\alpha=1$ since,
\begin{align*}
\PE[\param]{\Bbar{t}{\mu}{\param} }[\mu]=\PE[\param]{\frac{\supnorm{\ewghtfuncf[\param]{Y_t}} }{\dens[\mu,t]{Y_t}[\param]}}[\mu]
=\int \sup_{x \in \rset} g^\param(x,y) \rmd y=D_2 \int \frac{1}{|y|} \rmd y =\infty\eqsp.
\end{align*}

We now turn to the proof of \eqref{eq:moment:A4}.
Note first that 
\[
\limsup_{|x| \to \infty} \sup_{\param \in \Param} \frac{\PE[\param]{V_\delta(X_1)}[x]}{V_\delta(x)} =0 \eqsp,
\]
and for any $M < \infty$,
\(
\sup_{|x| \leq M} \sup_{\param \in \Param} \PE[\param]{V_\delta(X_1)}[x] < \infty .
\)
Therefore, there exist constants $\lambda_\delta \in \ooint{0,1}$ and $b_\delta < \infty$ such that, for all $x \in \Xset$,
\begin{equation}
\label{eq:drift-stochastic-volatility}
\sup_{\param \in \Param} \PE[\param]{V_\delta(X_1)}[x] \leq \lambda_\delta V_\delta(x) + b_\delta \eqsp.
\end{equation}
Analogously to \eqref{eq:moment-control-additive-noise}, this implies that, for all $\delta>0$, 
\begin{equation}
\label{eq:nuit-de-sumatra}
\sup_{t \in \nset} \sup_{\param \in \Param} \PE[\param]{V_\delta(X_t)}[\mu] \leq \mu(V_\delta) + b_\delta / (1-\lambda_\delta) < \infty \eqsp. 
\end{equation}
Using \eqref{eq:definition-weightfunc-SV} and \eqref{eq:borne-util-SV-2},
we get
\begin{equation}
\label{eq:borne-B00alpha}
\PE[\param]{ (\Bbar{t}{\mu}{\param} )^\alpha}[\mu] = \int \supnorm[\alpha]{\ewghtfuncf[\param]{y_t}} \{\dens[\mu,t]{y_t}[\param] \}^{1-\alpha} \rmd y_t
\leq D_2^\alpha \int |y_t|^{-\alpha} \{\dens[\mu,t]{y_t}[\param] \}^{1-\alpha} \rmd y_t \eqsp.
\end{equation}
We apply \autoref{lem:holder-inequality} to establish a bound for \eqref{eq:borne-B00alpha}. Consider the functions $\varphi$ and $\psi$ given by
\begin{align}
\label{eq:definition-psi}
&\psi(y) = 1/ |y| \eqsp, \\
\label{eq:definition-phi}
&\varphi(y)= \frac{|y|^\gamma}{|y|^2 \vee 1} \eqsp, \quad \text{with} \quad \frac{\gamma \alpha}{1 -\alpha} < 1 \eqsp, \quad 0 < \gamma < 1 \eqsp.
\end{align}
With these definitions, we get
\begin{equation*}
\int \varphi(y) \psi(y) \rmd y = \int \frac{1}{|y|} \frac{|y|^\gamma}{|y|^2 \vee 1} \rmd y  < \infty \ \, ,
\end{equation*}
showing that the first condition in \eqref{eq:conditions-holder-1} is satisfied. We now check \eqref{eq:conditions-holder-2}:
\begin{equation}
\label{eq:what-we-need}
\int \varphi^{-\frac{\alpha}{1-\alpha}}(y_t) \dens[\mu,t]{y_t}[\param] \rmd y_t
= \PE[\param]{\PE[\param]{\varphi^{-\frac{\alpha}{1-\alpha}}(Y_0)}[X_t]}[\mu]\eqsp.
\end{equation}
Since $Y_0 = \beta\exp(X_0/2)U_0$ it follows that
 $\PE[\param]{\varphi^{-\frac{\alpha}{1-\alpha}}(Y_0)}[x]= \PE{\varphi^{-\frac{\alpha}{1-\alpha}}(\beta \rme^{x/2} U)}$ where
$U$ is standard normal. We have
\begin{align*}
&\PE{\left( \frac{\beta^2 \rme^{x} U^2 \vee 1}{\beta^\gamma \rme^{\gamma x/2} |U|^\gamma} \right)^{\frac{\alpha}{1-\alpha}}} \\
&\qquad\leq \PE{(\beta \rme^{x/2} |U|)^{-\frac{\gamma \alpha}{1-\alpha}} \1_{\{\beta |U| \rme^{x/2} \leq 1\}}} +
\PE{(\beta^2 \rme^{x} U^2)^{\frac{\alpha}{1-\alpha}} \1_{\{\beta |U| \rme^{x/2} > 1\}}} \\
&\qquad\leq (\beta \rme^{\frac{x}{2}})^{-\frac{\gamma \alpha}{1-\alpha}} \PE{|U|^{-\frac{\gamma \alpha}{1-\alpha}}} +
(\beta^2 \rme^{x})^\frac{\alpha}{1-\alpha} \PE{|U|^{\frac{2 \alpha}{1-\alpha}}}  \eqsp.
\end{align*}
Since $\gamma \alpha/(1-\alpha) < 1$ it holds that $\PE{|U|^{-\frac{\gamma \alpha}{1-\alpha}}} < \infty$ and, additionally,
$\PE{|U|^{\frac{2 \alpha}{1-\alpha}}} < \infty$. Therefore, there exist constants $D_3 < \infty$ and $\delta > 0$ such that,
 for all $x \in \rset$ and $\param \in \Param$,
\begin{equation}
\label{eq:laborne}
\PE[\param]{\varphi^{-\frac{\alpha}{1-\alpha}}(Y_0)}[x] \leq D_3 \rme^{\delta |x|} = D_3 V_\delta(x) \eqsp.
\end{equation}
Using \eqref{eq:what-we-need}, \eqref{eq:laborne} and \eqref{eq:nuit-de-sumatra} verifies the second condition in \eqref{eq:conditions-holder-2}.
\autoref{lem:holder-inequality} can thus be used to conclude that
$ \PE[\param]{ (\Bbar{t}{\mu}{\param})^\alpha}[\mu] < \infty$ for all $\alpha \in \ooint{0,1}$.
Since this holds for any $t \in\nset$ and $\param\in\Param$, we establish the first part of \eqref{eq:moment:A4}.

We will now check that, for all $\alpha \in \ooint{0,1}$, $\PE[\param]{(\Cbar[t]{t+1}{\mu}{\param})^\alpha}[\mu] < \infty$.
Using \eqref{eq:borne-util-SV-1} and \eqref{eq:borne-util-SV-2}, we get
\begin{align}
\nonumber
\PE[\param]{(\Cbar[t]{t+1}{\mu}{\param})^\alpha}[\mu]
&= \PE[\param]{\frac{\supnorm[\alpha]{\ewghtfuncf[\param]{Y_t}} \left( \int g^\param(x_{t+1},Y_{t+1}) \rmd x_{t+1} \right)^\alpha}{ (\dens[\mu,t]{\chunk{Y}{t}{t+1}}[\param] )^\alpha}}\\
\nonumber
&= \iint \supnorm[\alpha]{\ewghtfuncf[\param]{y_t}} \left( \int g^\param(x_{t+1},y_{t+1}) \rmd x_{t+1} \right)^{\alpha} ( \dens[\mu,t]{\chunk{y}{t}{t+1}}[\param] )^{1-\alpha} \rmd \chunk{y}{t}{t+1} \eqsp, \\
&\leq D_1^\alpha D_2^\alpha \iint |y_t y_{t+1}|^{-\alpha} ( \dens[\mu,t]{\chunk{y}{t}{t+1}}[\param] )^{1-\alpha} \rmd \chunk{y}{t}{t+1} \eqsp.
\label{eq:borne-B01alpha}
\end{align}
We use again \autoref{lem:holder-inequality} with
\begin{equation}
\label{eq:definition-psi-1}
\psi(y_0,y_1) = |y_0|^{-1} |y_1|^{-1}
\end{equation}
and
\begin{equation}
\varphi(y_0,y_1) = \frac{|y_0|^\gamma |y_1|^\gamma}{(y_0^2 \vee 1) \, (y_1^2 \vee 1)} \eqsp,
\label{eq:definition-phi-1}
\end{equation}
with $\gamma \alpha / (1-\alpha) < 1$ and  $\gamma \in \ooint{0,1}$. Note first that
\begin{equation}
\label{eq:ile-de-java}
\iint \psi(y_0,y_1) \varphi(y_0,y_1)\rmd \chunk{y}{0}{1}
= \iint \left\{ (|y_0| |y_1|)^{1-\gamma} (y_0^2 \vee 1) (y_1^2 \vee 1) \right\}^{-1} \rmd \chunk{y}{0}{1} < \infty \eqsp.
\end{equation}
Hence, \eqref{eq:conditions-holder-1} is satisfied. We finally check \eqref{eq:conditions-holder-2}.
Using the conditional independence of the observations given the states and
\eqref{eq:laborne},
\begin{multline*}
\iint \varphi^{-\frac{\alpha}{1-\alpha}}(\chunk{y}{t}{t+1}) \dens[\mu,t]{\chunk{y}{t}{t+1}}[\param] \rmd \chunk{y}{t}{t+1}
= \PE[\param]{\CPEdoup[\mu]{\param}{\varphi^{-\frac{\alpha}{1-\alpha}}(\chunk{Y}{t}{t+1})}{ X_{t:t+1} }}[\mu] \\
= \PE[\param]{
\PE[\param]{\left( \frac{\beta^2 \rme^{X_0} U^2 \vee 1}{\beta \rme^{\gamma X_0/2} |U|^\gamma} \right)^{\frac{\alpha}{1-\alpha}}}[X_t]
\PE[\param]{\left( \frac{\beta^2 \rme^{X_0} U^2 \vee 1}{\beta \rme^{\gamma X_0/2} |U|^\gamma} \right)^{\frac{\alpha}{1-\alpha}}}[X_{t+1}]
}[\mu] \leq D_3^2 \PE[\param]{\rme^{\delta |X_t|} \rme^{\delta |X_{t+1}|}}[\mu] \eqsp.
\end{multline*}
Using \eqref{eq:drift-stochastic-volatility}, we get, from the Cauchy-Schwarz inequality,
\begin{align*}
\PE[\param]{\rme^{\delta |X_t|} \rme^{\delta |X_{t+1}|}}[\mu] \leq \left(\PE[\param]{\rme^{2 \delta |X_t|}}[\mu]\PE[\param]{\rme^{2\delta |X_{t+1}|}}[\mu] \right)^{1/2} \eqsp,
\end{align*}
Applying \eqref{eq:nuit-de-sumatra} with $\delta$ replaced by $2\delta$ yields \eqref{eq:conditions-holder-2}.
Using \autoref{lem:holder-inequality} thus establishes \eqref{eq:moment:A4} and thereby,  \ref{assum:b-moment-bound} holds.

Provided that $\theta_T$ converges to $\tparam$ at a rate $1/\sqrt{T}$ (see \autoref{rem:kullback}),
we may therefore apply \autoref{prop:N-depend-on-T} which shows that for any $\gamma \in \ooint{0,1}$,
 $\{\mineps{T}{N_T}^{-1}(\param_T) \}_{T \geq 1}$ is tight with $N_T= T^{1/\gamma}$.

\section{Proof of \autoref{theo:doeblin-condition-PG}}
\label{sec:proof}

We will now turn to the proof of the minorization condition in \autoref{theo:doeblin-condition-PG}.
As in the statement of the theorem, we will not explicitly indicate any possible dependence
on unknown model parameters in the notation in this section.
This is done for notational convenience and is without loss of generality.
Throughout this section, $\PP{}$ and $\PE{}$ refer to probability and expectation, respectively,
\wrt\ the random variables generated by the \pg algorithm.
The proof is inductive and follows from a series of lemmas.

\begin{lemma}\label{lem:jensens}
  Let $X\geq0$ and $Y>0$ be independent random variables. Then,
  \begin{equation*}
    \PE{\frac{X}{Y}}\geq  \frac{ \PE{X}}{ \PE{Y}} \eqsp.
  \end{equation*}
\end{lemma}
\begin{proof}
  Since $f(y) = 1/y$ is convex on $y > 0$ the result follows by independence and Jensen's inequality.
\end{proof}

\begin{lemma}
\label{lem:singapore}
Let $f$ and $h$ be nonnegative measurable functions. For $t \in \{0,\dots,T-1\}$, we have
\begin{multline}
\label{eq:bound-fond-1}
\CPE{\frac{\sum_{i=1}^N \ewght{t+1}{i} f(\epart{0:t+1}{i})}{\sum_{i=1}^N \ewght{t+1}{i} h(\epart{t+1}{i})}}{\mcff{t}{N}} \\
\geq
\frac{\sum_{i=1}^N \ewght{t}{i} \int \Kun{Y_{t+1}}{\epart{t}{i}}{\rmd x_{t+1}} f(\epart{0:t}{i},x_{t+1})}
{\sum_{i=1}^N \ewght{t}{i} \left[ \frac{N-2}{N-1} \Kunf{Y_{t+1}} h(\epart{t}{i}) + \frac{2}{N-1} \sup_{(x,x')} \ewghtfunc{Y_{t+1}}{x}{x'} h(x') \right]}  \eqsp,
\end{multline}
and
\begin{equation}
\label{eq:bound-fond-2}
\PE{\frac{\sum_{i=0}^N \ewght{0}{i} f(\epart{0}{i})}{\sum_{i=0}^N \ewght{0}{i} h(\epart{0}{i})}}
\geq  \frac{(N-1)\Xinit{g(\cdot,Y_0) f(\cdot)}}{(N-2) \Xinit{g(\cdot,Y_0) h(\cdot)} +2 \sup_x [\ewghtfuncz{Y_0}{x} h(x)]} \eqsp.
\end{equation}
\end{lemma}
\begin{proof}
Using that
$$
\ewght{t+1}{1} h(\epart{t+1}{1})+\ewght{t+1}{N} h(\epart{t+1}{N}) \leq 2 \sup_{(x,x')} \ewghtfunc{Y_{t+1}}{x}{x'} h(x') \eqsp,
$$
and that the weighted particles $\{( \epart{t+1}{i}, \ewght{t+1}{i}) \}_{i=1}^{N-1}$ are conditionally \iid\ \wrt\ $\mcff{t}{N}$, we get
\begin{align}
\nonumber
&\CPE{\frac{\sum_{i=1}^N \ewght{t+1}{i} f(\epart{0:t+1}{i})}{\sum_{i=1}^N \ewght{t+1}{i} h(\epart{t+1}{i})}}{\mcff{t}{N}}
\geq \CPE{\frac{\sum_{i=1}^{N-1} \ewght{t+1}{i} f(\epart{0:t+1}{i})}{\sum_{i=1}^N \ewght{t+1}{i} h(\epart{t+1}{i})}}{\mcff{t}{N}} \\  \nonumber
&\qquad\geq (N-1) \CPE{\frac{\ewght{t+1}{1} f(\epart{0:t+1}{1})}{\sum_{i=2}^{N-1} \ewght{t+1}{i} h(\epart{t+1}{i}) + 2 \sup_{(x,x')} \ewghtfunc{Y_{t+1}}{x}{x'} h(x')}}{\mcff{t}{N}} \\
 \label{eq:lower-bound}
&\qquad\geq (N-1) \frac{\CPE{\ewght{t+1}{1} f(\epart{0:t+1}{1})}{\mcff{t}{N}}}
{\CPE{\sum_{i=2}^{N-1} \ewght{t+1}{i} h(\epart{t+1}{i})}{\mcff{t}{N}}+ 2 \sup_{(x,x')} \ewghtfunc{Y_{t+1}}{x}{x'} h(x')} \eqsp,
\end{align}
where the last inequality follows from \autoref{lem:jensens}. Consider first the numerator in the \rhs\ of \eqref{eq:lower-bound}. We have
\begin{align}
  \nonumber
\CPE{\ewght{t+1}{1} f(\epart{0:t+1}{1})}{\mcff{t}{N}} 
&=  \frac{1}{ \sum_{l=1}^N \ewght{t}{l}  } \sum_{j=1}^N \ewght{t}{j} \int \Kis{Y_{t+1}}{\epart{t}{j}}{\rmd x_{t+1}} \ewghtfunc{Y_{t+1}}{\epart{t}{j}}{x_{t+1}} f(\epart{0:t}{j},x_{t+1}) \\ \label{eq:condexp}
&= \frac{1}{ \sum_{l=1}^N \ewght{t}{l}  } \sum_{j=1}^N \ewght{t}{j} \int \Kun{Y_{t+1}}{\epart{t}{j}}{\rmd x_{t+1}}  f(\epart{0:t}{j},x_{t+1}) \eqsp.
\end{align}
We now consider the denominator  in the \rhs\ of \eqref{eq:lower-bound}:
\begin{align*}
\CPE{\sum_{i=2}^{N-1} \ewght{t+1}{i} h(\epart{t+1}{i})}{\mcff{t}{N}} &= (N-2) \CPE{\ewght{t+1}{1} h(\epart{t+1}{1})}{\mcff{t}{N}} \\
&= (N-2) \frac{1}{ \sum_{l=1}^N \ewght{t}{l}  } \sum_{j=1}^N \ewght{t}{j} \Kunf{Y_{t+1}} h(\epart{t}{j}) \eqsp,
\end{align*}
where the last identity follows from \eqref{eq:condexp} with $f(\chunk{x}{0}{t+1})= h(x_{t+1})$. The proof of \eqref{eq:bound-fond-1} follows.

Consider now \eqref{eq:bound-fond-2}. Since the particles $\{\epart{0}{i}\}_{i=1}^{N-1}$ are \iid, we obtain, using again \autoref{lem:jensens},
\begin{align*}
\PE{\frac{\sum_{i=1}^N \ewght{0}{i} f(\epart{0}{i})}{\sum_{i=1}^N \ewght{0}{i} h(\epart{0}{i})}}
&\geq (N-1) \PE{\frac{\ewght{0}{1} f(\epart{0}{1})}{\sum_{i=2}^{N-1} \ewght{0}{i} h(\epart{0}{i}) + 2 \sup_x \ewghtfuncz{Y_0}{x} h(x)}}\\
&\geq \frac{(N-1) \PE{\ewght{0}{1} f(\epart{0}{1})}}{\PE{\sum_{i=2}^{N-1} \ewght{0}{i} h(\epart{0}{i})} + 2 \sup_x \ewghtfuncz{Y_0}{x} h(x)} \eqsp.
\end{align*}
The numerator is given by
\[
\PE{\ewght{0}{1} f(\epart{0}{1})}= \int \Xinitis{Y_0}{\rmd x_0} \ewghtfuncz{Y_0}{x_0} f(x_0) = \int \Xinit{\rmd x_0} g(x_0,Y_0) f(x_0) \eqsp.
\]
Similarly, we get
\begin{align*}
\PE{\sum_{i=2}^{N-1} \ewght{0}{i} h(\epart{0}{i})} = (N-2) \PE{\ewght{0}{1} h(\epart{0}{1})} 
=(N-2) \int \Xinit{\rmd x_0}  g(x_0,Y_0) h(x_0) \eqsp.
\end{align*}

\end{proof}

Define a sequence of nonnegative scalars $\{\beta_t\}_{t=0}^T$ by the backward recursion: $\beta_T=\supnorm{\ewghtfuncf{Y_T}}$, and for $t = T-1,T-2,\dots,0$,
\begin{multline}
\label{eq:def-M-t}
\beta_t = \supnorm{\ewghtfuncf{Y_t}} \left\{ \frac{2}{N-1} \sum_{\ell=1}^{T-t} \left( \frac{N-2}{N-1}\right)^{\ell-1} \beta_{t+\ell} \, \supnorm{\Kunf{\chunk{Y}{t+1}{t+\ell-1}}\bigone} \right. \\
\left. + \left( \frac{N-2}{N-1} \right)^{T-t} \supnorm{\Kunf{\chunk{Y}{t+1}{T}} \bigone} \right\}  \eqsp.
\end{multline}
Given  $\{\beta_t\}_{t=0}^T$, define the functions $\{ h_t \}_{t=0}^T$, $h_t: \Xset \to \rset_+$, by the backward recursion: $h_T = \bigone$ and, for all $t= T-1, T-2,\dots,0$,
\begin{equation}
\label{eq:recursive-def-h-t}
h_t: x \mapsto h_t(x) = \frac{N-2}{N-1} \; \Kunf{Y_{t+1}} h_{t+1}(x) + \frac{2}{N-1} \beta_{t+1} \eqsp.
\end{equation}
By solving the backward recursion, \eqref{eq:recursive-def-h-t} implies
\begin{multline}
\label{eq:def-h-t}
h_t(x) =  \frac{2}{N-1} \sum_{\ell=1}^{T-t} \left( \frac{N-2}{N-1} \right)^{\ell-1} \beta_{t+\ell} \Kunf{\chunk{Y}{t+1}{t+\ell-1}} \bigone(x) 
+
\left( \frac{N-2}{N-1} \right)^{T-t} \Kunf{\chunk{Y}{t+1}{T}}\bigone(x) \eqsp.
\end{multline}
For $D \in \Xsigma^{\otimes (T+1)}$, set $f_T(\chunk{x}{0}{T})= \1_D(\chunk{x}{0}{T})$ and, for $t = T-1, T-2, \dots, 0$,
\begin{equation}
\label{eq:recursive-def-f-t}
f_t(\chunk{x}{0}{t})= \int \Kun{Y_{t+1}}{x_t}{\rmd x_{t+1}} f_{t+1}(\chunk{x}{0}{t+1}) \eqsp,
\end{equation}
or equivalently,
\begin{equation}
\label{eq:def-f-t}
f_t(\chunk{x}{0}{t})= \int \prod_{\ell=1}^{T-t} \Kun{Y_{t+\ell}}{x_{t+\ell-1}}{\rmd x_{t+\ell}} \1_D(\chunk{x}{0}{T}) \eqsp.
\end{equation}

\begin{lemma}
For any $D \in \Xsigma^{\otimes (T+1)}$,
\begin{equation*}
\PE{\frac{\sum_{i=1}^N \ewght{T}{i} \1_D(\epart{0:T}{i})}{\sum_{i=1}^N \ewght{T}{i}}} \geq  \frac{(N-1)\dens[\Xinitv]{\chunk{Y}{0}{T}}}{(N-2)\Xinit{g(\cdot,Y_0) h_0(\cdot)}+2\beta_0} \post{\Xinitv,0:T}{\chunk{Y}{0}{T}}(D) \eqsp.
\end{equation*}
\end{lemma}
\begin{proof}
Note first that, by construction,
\begin{equation}
\label{eq:kuala-lumpur}
\PE{\frac{\sum_{i=1}^N \ewght{T}{i} \1_D(\epart{0:T}{i})}{\sum_{i=1}^N \ewght{T}{i}}} =
\PE{\frac{\sum_{i=1}^N \ewght{T}{i} f_T(\epart{0:T}{i})}{\sum_{i=1}^N \ewght{T}{i}  h_T(\epart{T}{i})}}
\eqsp.
\end{equation}
We now show that, by backward induction, for all $t \in \{0,\dots,T-1\}$,
\begin{equation}
\label{eq:malaisie}
\PE{\frac{\sum_{i=1}^N \ewght{t+1}{i} f_{t+1}(\epart{0:t+1}{i})}{\sum_{i=1}^N \ewght{t+1}{i} h_{t+1}(\epart{t+1}{i})}}
\geq
\PE{\frac{\sum_{i=1}^N \ewght{t}{i} f_t(\epart{0:t}{i})}{\sum_{i=1}^N \ewght{t}{i} h_t(\epart{t}{i})}} \eqsp.
\end{equation}
To obtain \eqref{eq:malaisie}, note first that the tower property of the conditional expectation, \autoref{lem:singapore}, and
\eqref{eq:recursive-def-f-t} imply
\begin{align*}
&\PE{\frac{\sum_{i=1}^N \ewght{t+1}{i} f_{t+1}(\epart{0:t+1}{i})}{\sum_{i=1}^N \ewght{t+1}{i} h_{t+1}(\epart{t+1}{i})}}
= \PE{\CPE{\frac{\sum_{i=1}^N \ewght{t+1}{i} f_{t+1}(\epart{0:t+1}{i})}{\sum_{i=1}^N \ewght{t+1}{i} h_{t+1}(\epart{t+1}{i})}}{\mcff{t}{N}}}\\
&\qquad\geq  \PE{\frac{\sum_{i=1}^N \ewght{t}{i} f_t(\epart{0:t}{i})}{\sum_{i=1}^N \ewght{t}{i} \left[ \frac{N-2}{N-1} \Kunf{Y_{t+1}} h_{t+1}(\epart{t}{i}) + \frac{2}{N-1} \sup_{(x,x')} \ewghtfunc{Y_{t+1}}{x}{x'} h_{t+1}(x') \right]}} \eqsp.
\end{align*}
By the triangle inequality, it follows directly from \eqref{eq:def-M-t} and \eqref{eq:def-h-t} that
\begin{align}
\label{eq:bound-beta0}
&\sup_x \ewghtfuncz{Y_0}{x} h_0(x) \leq \beta_0 \eqsp, \\
\label{eq:bound-betat}
&\sup_{x,x'} \ewghtfunc{Y_{t+1}}{x}{x'} h_{t+1}(x') \leq \beta_{t+1} \eqsp, \quad t \in \{0, \dots, T-1 \} \eqsp.
\end{align}
Combining the inequality \eqref{eq:bound-betat} with the definition of $h_t$ in \eqref{eq:recursive-def-h-t} yields
\begin{equation*}
\sum_{i=1}^N \ewght{t}{i} \left[ \frac{N-2}{N-1} \Kunf{Y_{t+1}} h_{t+1}(\epart{t}{i}) + \frac{2}{N-1} \sup_{(x,x')} \ewghtfunc{Y_{t+1}}{x}{x'} h_{t+1}(x') \right] 
\leq \sum_{i=1}^N \ewght{t}{i} h_t(\epart{t}{i}) \eqsp,
\end{equation*}
showing \eqref{eq:malaisie}. Combining \eqref{eq:malaisie} with \eqref{eq:kuala-lumpur} and using \autoref{lem:singapore}-\eqref{eq:bound-fond-2}
establishes that
\begin{align*}
\PE{\frac{\sum_{i=1}^N \ewght{T}{i} \1_D(\epart{0:T}{i})}{\sum_{i=1}^N \ewght{T}{i}}}
\geq  \PE{\frac{\sum_{i=1}^N \ewght{0}{i} f_0(\epart{0}{i})}{\sum_{i=1}^N \ewght{0}{i}  h_0(\epart{0}{i})}}
\geq  \frac{(N-1)\Xinit{g(\cdot,Y_0) f_0(\cdot)}}{(N-2)\Xinit{g(\cdot,Y_0) h_0(\cdot)}+2 \beta_0} \eqsp,
\end{align*}
where the last inequality stems from \eqref{eq:bound-beta0}. The proof is completed by noting that
\[
\Xinit{g(\cdot,Y_0) f_0(\cdot)} = \post{\Xinitv,0:T}{\chunk{Y}{0}{T}}(D) \dens[\Xinitv]{\chunk{Y}{0}{T}} \eqsp. \]
\end{proof}

Finally, to prove \autoref{theo:doeblin-condition-PG} it remains to show the following.

\begin{lemma}
\label{lem:bound-denominator}
With $B_{t,T}$ defined as in \eqref{eq:def-b-T}, it holds that
\begin{equation}
\label{eq:bound-denominator}
(N-2)\Xinit{g(\cdot,Y_0) h_0(\cdot)}+2\beta_0
\leq (N-1) \dens[\Xinitv]{\chunk{Y}{0}{T}} \left[\prod_{t=0}^T \frac{2 B_{t,T}+N-2}{N-1} \right]\eqsp.
\end{equation}
\end{lemma}

\begin{proof}
Define for $t \in \{0,\dots,T\}$,
\begin{equation}
\label{eq:def-Q-t}
\alpha_t= \frac{\beta_t}{\cdens[\Xinitv]{\chunk{Y}{t}{T}}{\chunk{Y}{0}{t-1}}} \eqsp,
\end{equation}
with the convention $\cdens[\Xinitv]{\chunk{Y}{0}{T}}{\chunk{Y}{0}{ -1 }}=\dens[\Xinitv]{\chunk{Y}{0}{T}}$. In particular,
$\alpha_0= \beta_0/\dens[\Xinitv]{\chunk{Y}{0}{T}}$.

Eq.~\eqref{eq:def-M-t} implies
\begin{multline*}
\alpha_t = \supnorm{\ewghtfuncf{Y_t}}
\left\{ \frac{2}{N-1} \sum_{\ell=1}^{T-t} \left( \frac{N-2}{N-1} \right)^{\ell-1} \alpha_{t+\ell} \left[ \frac{\supnorm{\Kunf{\chunk{Y}{t+1}{t+\ell-1}}\bigone} \cdens[\Xinitv]{\chunk{Y}{t+\ell}{T}}{\chunk{Y}{0}{t+\ell-1}}}{\cdens[\Xinitv]{\chunk{Y}{t}{T}}{\chunk{Y}{0}{t-1}}} \right] \right. \\
\left. + \left( \frac{N-2}{N-1} \right)^{T-t} \frac{\supnorm{\Kunf{\chunk{Y}{t+1}{T}}\bigone}}{\cdens[\Xinitv]{\chunk{Y}{t}{T}}{\chunk{Y}{0}{t-1}}} \right\} \eqsp.
\end{multline*}
The identity
\[
\frac{\cdens[\Xinitv]{\chunk{Y}{t+\ell}{T}}{\chunk{Y}{0}{t+\ell-1}}}{\cdens[\Xinitv]{\chunk{Y}{t}{T}}{\chunk{Y}{0}{t-1}}}= \frac{1}{\cdens[\Xinitv]{\chunk{Y}{t}{t+\ell-1}}{\chunk{Y}{0}{t-1}}}
\]
and the definition in \eqref{eq:def-b-T} imply that
\begin{equation}
\label{eq:ineq-alpha}
\alpha_t \leq B_{t,T} \left\{ \frac{2}{N-1} \sum_{\ell=1}^{T-t} \left( \frac{N-2}{N-1} \right)^{\ell-1} \alpha_{t+\ell} + \left( \frac{N-2}{N-1} \right)^{T-t} \right\} \eqsp.
\end{equation}
By a backward induction, define the sequence $\{\tilde{\alpha}_t \}_{t=0}^T$ as follows: set $\tilde{\alpha}_T = B_{T,T}$ and
\begin{equation}
\label{eq:def-tilde-alpha}
\tilde{\alpha}_t= B_{t,T} \left[ \frac{2}{N-1} \sum_{\ell=1}^{T-t} \left( \frac{N-2}{N-1} \right)^{\ell-1} \tilde{\alpha}_{t+\ell} + \left( \frac{N-2}{N-1} \right)^{T-t} \right] \eqsp.
\end{equation}

Since by construction, $\alpha_T \leq B_{T,T} = \tilde{\alpha}_T$, an elementary backward recursion using \eqref{eq:ineq-alpha} shows that,
\begin{equation}
\label{eq:tilde-alpha-bound-alpha}
\text{for all $t\in \{0,\dots,T\}$, $\alpha_t \leq \tilde{\alpha}_t$} \eqsp.
\end{equation}
However
\begin{align*}
\tilde{\alpha}_{t-1} &= B_{t-1,T} \left[ \frac{2}{N-1} \sum_{s=1}^{T-t+1} \left( \frac{N-2}{N-1} \right)^{\ell-1} \tilde{\alpha}_{t+\ell-1} + \left( \frac{N-2}{N-1} \right)^{T-t+1}  \right]\\
&= B_{t-1,T} \left[ \frac{2}{N-1} \tilde{\alpha}_t + \frac{2}{N-1} \sum_{k=1}^{T-t} \left( \frac{N-2}{N-1} \right)^{k} \tilde{\alpha}_{t+k} + \left( \frac{N-2}{N-1} \right)^{T-t+1}  \right] \\
& = \frac{2 B_{t-1,T}}{N-1} \tilde{\alpha}_t + B_{t-1,T} \frac{N-2}{N-1} \frac{\tilde{\alpha}_t}{B_{t,T}} \\
& = \frac{B_{t-1,T}}{B_{t,T}} \left[ \frac{2 B_{t,T}}{N-1} + \frac{N-2}{N-1} \right] \tilde{\alpha}_t \eqsp.
\end{align*}
Therefore
\begin{equation}
\label{eq:tilde-alpha-value}
\tilde{\alpha}_0 = B_{0,T} \prod_{t=1}^{T} \frac{2 B_{t,T} + N -2}{N-1} \eqsp.
\end{equation}
Now, since
\begin{equation*}
h_0(x) =  \frac{2}{N-1} \sum_{s=1}^{T} \left( \frac{N-2}{N-1} \right)^{s-1} \beta_{s} \Kunf{\chunk{Y}{1}{s-1}} \bigone(x) \\
+
\left( \frac{N-2}{N-1} \right)^{T} \Kunf{\chunk{Y}{1}{T}}\bigone(x) \eqsp,
\end{equation*}
we have
\begin{equation*}
\Xinit{g(\cdot,Y_0) h_0(\cdot)}=  \frac{2}{N-1} \sum_{s=1}^{T} \left( \frac{N-2}{N-1} \right)^{s-1} \beta_{s} \dens[\Xinitv]{\chunk{Y}{0}{s-1}}
+ \left( \frac{N-2}{N-1} \right)^{T} \dens[\Xinitv]{\chunk{Y}{0}{T}} \eqsp.
\end{equation*}
Plugging \eqref{eq:def-Q-t} into this equation and using that
$$
\dens[\Xinitv]{\chunk{Y}{0}{s-1}}
=\frac{  \dens[\Xinitv]{\chunk{Y}{0}{T}}  }{ \cdens[\Xinitv]{\chunk{Y}{s}{T}}{\chunk{Y}{0}{s-1}} }
$$
yield
\begin{align*}
\Xinit{g(\cdot,Y_0) h_0(\cdot)}=  \frac{2}{N-1} \sum_{s=1}^{T} \left( \frac{N-2}{N-1} \right)^{s-1} \alpha_{s}
\dens[\Xinitv]{\chunk{Y}{0}{T}}
+
\left( \frac{N-2}{N-1} \right)^{T} \dens[\Xinitv]{\chunk{Y}{0}{T}} \eqsp.
\end{align*}
Finally, using \eqref{eq:tilde-alpha-bound-alpha} and then \eqref{eq:def-tilde-alpha},
\begin{align*}
&(N-2)\Xinit{g(\cdot,Y_0) h_0(\cdot)}+2 \beta_0\\
&\quad \leq\dens[\Xinitv]{\chunk{Y}{0}{T}}  \left\{ (N-2)\left[\frac{2}{N-1} \sum_{s=1}^{T} \left( \frac{N-2}{N-1} \right)^{s-1} \alpha_{s}+\left(\frac{N-2}{N-1} \right)^{T}\right]  +2 \alpha_0\right\}\\
&\quad \leq \dens[\Xinitv]{\chunk{Y}{0}{T}}  \left\{ (N-2)\left[\frac{2}{N-1} \sum_{s=1}^{T} \left( \frac{N-2}{N-1} \right)^{s-1} \tilde{\alpha}_{s}+\left(\frac{N-2}{N-1} \right)^{T}\right]  +2 \tilde{\alpha}_0\right\}\\
&\quad \leq \dens[\Xinitv]{\chunk{Y}{0}{T}} \left((N-2)  \frac{\tilde{\alpha}_0}{B_{0,T}}+2 \tilde{\alpha}_0 \right)\\
&\quad =\dens[\Xinitv]{\chunk{Y}{0}{T}}   (N-2+2 B_{0,T})\prod_{t=1}^{T} \frac{2 B_{t,T} + N -2}{N-1} \eqsp,
\end{align*}
where the last equality follows from \eqref{eq:tilde-alpha-value}. The proof follows.
\end{proof}

\section{Proof of \autoref{prop:N-depend-on-T}}
\label{sec:proof:prop:N-depend-on-T}

Define
\begin{align}
\label{eq:def-hat-B-nu-0}
&\Bhat[t]{t+\ell}{\mu}{\param}\eqdef
\frac{\supnorm{\ewghtfuncf[\param]{Y_t}} \supnorm{\Kunf[\param]{\chunk{Y}{t+1}{t+\ell}} \bigone}}{\cdens[\mu]{\chunk{Y}{t}{t+\ell}}{\chunk{Y}{0}{t-1}}[\param]}
\eqsp, \\
&\Chat[t]{t+\ell}{\mu}{\param} \eqdef \frac{\supnorm{\ewghtfuncf[\param]{Y_t}} \int \lambda(\rmd x_{t+1}) g^\param(x_{t+1},Y_{t+1}) \Kunf[\param]{\chunk{Y}{t+2}{t+\ell}} \bigone(x_{t+1})}{\cdens[\mu]{\chunk{Y}{t}{t+\ell}}{\chunk{Y}{0}{t-1}}[\param]}
\eqsp. \label{eq:def-check-b-t-ell}
\end{align}
Note that
\begin{align}
\Bhat[t]{t+\ell}{\mu}{\param} &= \Bbar[t]{t+\ell}{\mu}{\param} \frac{ \dens[\mu,t]{\chunk{Y}{t}{t+\ell}}[\param] }{ \cdens[\mu]{\chunk{Y}{t}{t+\ell}}{\chunk{Y}{0}{t-1}}[\param] }
&&\text{and} &
\Chat[t]{t+\ell}{\mu}{\param} &= \Cbar[t]{t+\ell}{\mu}{\param} \frac{ \dens[\mu,t]{\chunk{Y}{t}{t+\ell}}[\param] }{ \cdens[\mu]{\chunk{Y}{t}{t+\ell}}{\chunk{Y}{0}{t-1}}[\param] }
\eqsp,
\label{eq:rel-bar-B-hat-B}
\end{align}
where $\Bbar[t]{t+\ell}{\mu}{\param}$ and $\Cbar[t]{t+\ell}{\mu}{\param}$ are defined in \eqref{eq:def-bar-b-t-ell} and
\eqref{eq:def-bar-c-t-ell}, respectively.
\begin{lemma} \label{lem:tilde-B-martingale}
For all $\param \in \Param$, the sequence $\{\Chat[t]{t+\ell}{\mu}{\param}\}_{\ell \geq 0}$ defined in \eqref{eq:def-check-b-t-ell} is a $(\PP^\param_\mu,\{\mcf_{t+\ell}\}_{\ell \geq 0})$-martingale, where $\mcf_{t}=\sigma(\chunk{Y}{0}{t})$.
\end{lemma}
\begin{proof}
For all $\ell \geq 0$, 
\begin{multline*}
 \CPEdoup[\mu]{\param}{\Chat[t]{t+\ell+1}{\mu}{\param}}{\mcf_{t+\ell}}
= \supnorm{\ewghtfuncf[\param]{Y_t}} \int
\Bigg\{ \cdens[\mu]{y_{t+\ell+1}}{\chunk{Y}{0}{t+\ell}}[\param] \\
\times  \frac{ \int \lambda(\rmd x_{t+1}) g^\param(x_{t+1},Y_{t+1}) \Kun[\param]{\chunk{Y}{t+2}{t+\ell}}{x_{t+1}}{\rmd x_{t+\ell}}  \Kunf[\param]{y_{t+\ell+1}}\bigone(x_{t+\ell})}{\cdens[\mu]{\chunk{Y}{t}{t+\ell},y_{t+\ell+1}}{\chunk{Y}{0}{t-1}}[\param]}
  \kappa(\rmd y_{t+\ell+1}) \Bigg\}
\eqsp.
\end{multline*}
Combining this identity with
$$
\cdens[\mu]{\chunk{Y}{t}{t+\ell},y_{t+\ell+1}}{\chunk{Y}{0}{t-1}}[\param]=\cdens[\mu]{y_{t+\ell+1}}{\chunk{Y}{0}{t+\ell}}[\param]
\cdens[\mu]{\chunk{Y}{t}{t+\ell}}{\chunk{Y}{0}{t-1}}[\param] \eqsp,
$$
and $\int \Kunf[\param]{y_{t+\ell+1}}\bigone(x_{t+\ell}) \kappa(\rmd y_{t+\ell+1})=M^\param(x_{t+\ell},\Xset)=1$, we obtain
\begin{multline*}
\CPEdoup[\mu]{\param}{\Chat[t]{t+\ell+1}{\mu}{\param}}{\mcf_{t+\ell}} \\
= \frac{\supnorm{\ewghtfuncf[\param]{Y_t}} \int \lambda(\rmd x_{t+1}) g^\param(x_{t+1},Y_{t+1}) \Kunf[\param]{\chunk{Y}{t+2}{t+\ell}}\bigone(x_{t+1})  }{\cdens[\mu]{\chunk{Y}{t}{t+\ell}}{\chunk{Y}{0}{t-1}}[\param]}=\Chat[t]{t+\ell}{\mu}{\param}\eqsp,
\end{multline*}
which completes the proof.
\end{proof}

\begin{lemma} \label{lem:bound:NTU}
For all $0\leq \gamma<1$ and all $\ell  \in \nset$,
\begin{align}
&\PE[\param]{(\Bhat[t]{t+\ell}{\mu}{\param})^\gamma}[\mu]\leq \PE[\param]{(\Bbar[t]{t+\ell}{\mu}{\param})^\gamma}[\mu]\eqsp, \label{eq:ntu}\\
&\PE[\param]{(\Chat[t]{t+\ell}{\mu}{\param})^\gamma}[\mu]\leq \PE[\param]{(\Cbar[t]{t+\ell}{\mu}{\param})^\gamma }[\mu]\eqsp.\label{eq:nus}
\end{align}
\end{lemma}
\begin{proof}
Using \eqref{eq:rel-bar-B-hat-B}, the proof of \eqref{eq:ntu} and \eqref{eq:nus} follow from the inequality:
\begin{equation} \label{eq:ineg-b}
\PE[\param]{\frac{\psi(\chunk{Y}{t}{t+\ell})}{\left\{\cdens[\mu]{\chunk{Y}{t}{t+\ell}}{\chunk{Y}{0}{t-1}}[\param]\right\}^\gamma}}[\mu]
\leq
\PE[\param]{\frac{\psi(\chunk{Y}{t}{t+\ell})}{\left\{\dens[\mu,t]{\chunk{Y}{t}{t+\ell}}[\param]\right\}^\gamma}}[\mu]\eqsp,
\end{equation}
which holds for any nonnegative measurable function $\psi: \Yset^{\ell+1} \to \rset^+$. We now show \eqref{eq:ineg-b}. Note first that, by applying the tower property of the conditional expectation and then the Tonelli-Fubini theorem, we get
\begin{align}
\nonumber
\PE[\param]{\frac{\psi(\chunk{Y}{t}{t+\ell})}{\left\{\cdens[\mu]{\chunk{Y}{t}{t+\ell}}{\chunk{Y}{0}{t-1}}[\param]\right\}^\gamma}}[\mu]
&= \PE[\param]{\CPEdoup[\mu]{\param}{ \frac{\psi(\chunk{Y}{t}{t+\ell})}{\left\{\cdens[\mu]{\chunk{Y}{t}{t+\ell}}{\chunk{Y}{0}{t-1}}[\param]\right\}^\gamma} }{\chunk{Y}{0}{t-1}} }[\mu] \\
\nonumber
&= \PE[\param]{\int \psi(\chunk{y}{t}{t+\ell}) \left\{\cdens[\mu]{\chunk{y}{t}{t+\ell}}{\chunk{Y}{0}{t-1}}[\param]\right\}^{1-\gamma}
  \kappa^{\otimes(\ell+1)}(\rmd y_{t:t+\ell})}[\mu] \\
  \label{eq:borne-B00alpha-1}
  &= \int \psi(\chunk{y}{t}{t+\ell}) \PE[\param]{\left\{\cdens[\mu]{\chunk{y}{t}{t+\ell}}{\chunk{Y}{0}{t-1}}[\param]\right\}^{1-\gamma}}[\mu]
  \kappa^{\otimes(\ell+1)}(\rmd y_{t:t+\ell}) \eqsp.
\end{align}
By the Jensen identity, $\PE[\param]{\left\{\cdens[\mu]{\chunk{y}{t}{t+\ell}}{\chunk{Y}{0}{t-1}}[\param]\right\}^{1-\gamma}}[\mu] \leq
\left\{ \PE[\param]{\cdens[\mu]{\chunk{y}{t}{t+\ell}}{\chunk{Y}{0}{t-1}}[\param]}[\mu] \right\}^{1-\gamma}$. On the other hand,
\begin{align}
  \PE[\param]{\cdens[\mu]{\chunk{y}{t}{t+\ell}}{\chunk{Y}{0}{t-1}}[\param]}[\mu]
  = \int \cdens[\mu]{\chunk{y}{t}{t+\ell}}{\chunk{y}{0}{t-1}}[\param] \dens[\mu]{\chunk{y}{0}{t-1}}[\param] \kappa^{\otimes t}(\rmd y_{0:t-1})
  = \dens[\mu,t]{\chunk{y}{t}{t+\ell}}[\param]
  \eqsp.
\end{align}
The proof of \eqref{eq:ineg-b} follows by combining the above relations.

\end{proof}
\begin{lemma} \label{lem:tilde-B-moment-alpha}
Assume \ref{assum:m-upper-bound} and \ref{assum:b-moment-bound}. Then, for all $0\leq \gamma<\alpha$,
$$
\sup_{t \geq 0} \sup_{\param  \in \Param}\PE[\param]{\left(\sup_{\ell \geq 0} \Bhat[t]{t+\ell}{\mu}{\param}\right)^\gamma}[\mu]<\infty\eqsp.
$$
where $\alpha$ is defined in \ref{assum:b-moment-bound}.
\end{lemma}
\begin{proof}
Under \ref{assum:m-upper-bound}, we obtain by definitions of $\Bhat[t]{t+\ell}{\mu}{\param}$
 and $\Chat[t]{t+\ell}{\mu}{\param}$,
$$
\sup_{\ell \geq 0} \Bhat[t]{\ell}{\mu}{\param} \leq \sum_{\ell=0}^{\ell_\star-1}
\Bhat[t]{\ell}{\mu}{\param}
+ \sup_{\ell \geq \ell_\star} \Bhat[t]{t+\ell}{\mu}{\param} \leq
 \sum_{\ell=0}^{\ell_\star-1}\Bhat[t]{t+\ell}{\mu}{\param}
+ \sigma_+ \sup_{\ell \geq \ell_\star} \Chat[t]{t+\ell}{\mu}{\param}\eqsp,
$$
where $\sigma_+$  and $\ell_\star$ are defined in \ref{assum:m-upper-bound} and \ref{assum:b-moment-bound}, respectively. Then, by subadditivity of $u \mapsto u^\gamma$,
$$
\PE[\param]{\left(\sup_{\ell \geq 0} \Bhat[t]{t+\ell}{\mu}{\param} \right)^\gamma}[\mu] \leq \sum_{\ell=0}^{\ell_\star-1}\PE[\param]{(\Bhat[t]{t+\ell}{\mu}{\param})^\gamma}[\mu]
+  (\sigma_+)^\gamma \PE[\param]{\sup_{\ell \geq \ell_\star} (\Chat[t]{t+\ell}{\mu}{\param})^\gamma}[\mu]\eqsp.
$$
Applying \autoref{lem:bound:NTU} and \eqref{eq:bound-moment-B}, it is thus sufficient to
bound $$\PE[\param]{ \sup_{\ell \geq \ell_\star} (\Chat[t]{t+\ell}{\mu}{\param})^\gamma}[\mu] \eqsp.$$
Since by \autoref{lem:tilde-B-martingale}, $\{\Chat[t]{t+\ell}{\mu}{\param}\}_{k\geq 0}$ is a
$\{\mcf_{t+\ell}\}_{\ell \geq 0}$-martingale and $\alpha \in (0,1)$, we have that $\{(\Chat[t]{t+\ell}{\mu}{\param})^\alpha\}_{\ell \geq 0}$ is a nonnegative $\{\mcf_{t+\ell}\}_{\ell \geq 0}$-supermartingale. The Doob maximal inequality then applies: for all $a>0$,
$$
a \CPPdoup[\mu]{\param}{\sup_{\ell \geq \ell_\star} (\Chat[t]{t+\ell}{\mu}{\param})^\alpha \geq a}{\mcf_{t+\ell_\star-1}}
\leq \CPEdoup[\mu]{\param}{(\Chat[t]{t+\ell_\star}{\mu}{\param})^\alpha}{\mcf_{t+\ell_\star-1}} \eqsp.
$$
Take now the expectation in both sides of the previous inequality and set $\delta=a^{\gamma/\alpha}$. We obtain
\begin{align*}
\PPdoup[\mu]{\param}{\sup_{\ell \geq \ell_\star} (\Chat[t]{t+\ell}{\mu}{\param})^\gamma \geq \delta}
\leq \delta^{-\alpha/\gamma}
\PE[\param]{(\Chat[t]{t+\ell_\star}{\mu}{\param})^\alpha}[\mu] \eqsp.
\end{align*}
Combining this with the inequality $\PE{U} \leq 1+ \int_{1}^\infty \PP\left[ U>\delta \right]\,\rmd \delta$ which holds for all nonnegative random variable $U$, we obtain under \ref{assum:b-moment-bound}
\begin{align*}
  \PE[\param]{\sup_{\ell \geq \ell_\star} (\Chat[t]{t+\ell}{\mu}{\param})^\gamma}[\mu]
  &\leq 1 + \left(\int_1^\infty \delta^{-\alpha/\gamma} \rmd \delta \right)
  \PE[\param]{(\Chat[t]{t+\ell_\star}{\mu}{\param})^\alpha }[\mu] \\
  &=1+\frac{\gamma}{\alpha-\gamma}\PE[\param]{(\Chat[t]{t+\ell_\star}{\mu}{\param})^\alpha}[\mu] \eqsp.
\end{align*}
The proof follows by applying again \autoref{lem:bound:NTU} under \ref{assum:b-moment-bound}.
\end{proof}

\begin{proof}[Proof of \autoref{prop:N-depend-on-T}]
For simplicity we will use in this proof the notations
$\densstat{\chunk{Y}{0}{t}}[\param] \eqdef \dens[\pi^\param]{\chunk{Y}{0}{t}}[\param]$ and
$\cdensstat{\chunk{Y}{t}{s}}{\chunk{Y}{0}{t-1}}[\param]= \cdens[\pi^\param]{\chunk{Y}{t}{s}}{\chunk{Y}{0}{t-1}}[\param]$.
First note that
\begin{align*}
\PPstat^\tparam \left\{\frac{\densstat{\chunk{Y}{0}{T}}[\tparam]}{\dens[\mu]{\chunk{Y}{0}{T}}[\param_T]}> \rho\right\}&=
\PPstat^\tparam \left\{\ln \frac{\densstat{\chunk{Y}{0}{T}}[\tparam]}{\dens[\mu]{\chunk{Y}{0}{T}}[\param_T]}+\frac{\dens[\mu]{\chunk{Y}{0}{T}}[\param_T]}{\densstat{\chunk{Y}{0}{T}}[\tparam]}-1
> \ln \rho+\frac{\dens[\mu]{\chunk{Y}{0}{T}}[\param_T]}{\densstat{\chunk{Y}{0}{T}}[\tparam]}-1\right\}\\
&\leq \PPstat^\tparam \left\{\ln \frac{\densstat{\chunk{Y}{0}{T}}[\tparam]}{\dens[\mu]{\chunk{Y}{0}{T}}[\param_T]}+\frac{\dens[\mu]{\chunk{Y}{0}{T}}[\param_T]}{\densstat{\chunk{Y}{0}{T}}[\tparam]}-1
> \ln \rho-1\right\} \eqsp.
\end{align*}
Now, since for all $u > 0$, $\ln(u)+u^{-1}-1\geq 0$, we obtain for for all $\rho>\rme= \exp(1)$:
\begin{align*}
\PPstat^\tparam \left\{\frac{\densstat{\chunk{Y}{0}{T}}[\tparam]}{\dens[\mu]{\chunk{Y}{0}{T}}[\param_T]} > \rho\right\}&\leq \frac{1}{\ln \rho-1}
\PEstat[\tparam]{\ln \frac{\densstat{\chunk{Y}{0}{T}}[\tparam]}{\dens[\mu]{\chunk{Y}{0}{T}}[\param_T]}+\frac{\dens[\mu]{\chunk{Y}{0}{T}}[\param_T]}{\densstat{\chunk{Y}{0}{T}}[\tparam]}-1} \\
&= \frac{1}{\ln \rho-1}
\PEstat[\tparam]{\ln \frac{\densstat{\chunk{Y}{0}{T}}[\tparam]}{\dens[\mu]{\chunk{Y}{0}{T}}[\param_T]}} \eqsp.
\end{align*}
This implies that for all $M>0$ and all $\rho>\rme$,
\begin{align}
\PPstat^\tparam(\mineps{T}{N_T}^{-1}(\param_T)>M) &\leq \PE[\param_T]{\frac{\densstat{\chunk{Y}{0}{T}}[\tparam]}{\dens[\mu]{\chunk{Y}{0}{T}}[\param_T]}
\1_{\left\{\frac{\densstat{\chunk{Y}{0}{T}}[\tparam]}{\dens[\mu]{\chunk{Y}{0}{T}}[\param_T]}\leq \rho \right\}}
\1_{\left\{\mineps{T}{N_T}^{-1}(\param_T)>M \right\}}}[\mu]+\PPstat^\tparam \left\{\frac{\densstat{\chunk{Y}{0}{T}}[\tparam]}{\dens[\mu]{\chunk{Y}{0}{T}}[\param_T]}> \rho\right\} \nonumber\\
&\leq \rho \PPdoup[\mu]{\param_T}{\mineps{T}{N_T}^{-1}(\param_T)>M}+ \frac{1}{\ln \rho-1} \PEstat[\tparam]{\ln \frac{\densstat{\chunk{Y}{0}{T}}[\tparam]}{\dens[\mu]{\chunk{Y}{0}{T}}[\param_T]}} \nonumber\\
&\leq \rho \PPdoup[\mu]{\param_T}{\mineps{T}{N_T}^{-1}(\param_T)>M}+ \frac{1}{\ln \rho-1}\left( \sup_{T \geq 0}\PEstat[\tparam]{\ln \frac{\densstat{\chunk{Y}{0}{T}}[\tparam]}{\dens[\mu]{\chunk{Y}{0}{T}}[\param_T]}} \right) \eqsp. \label{eq:one}
\end{align}
We consider first  the last term of the \rhs. Note first that, by the tower property,
\[
\PEstat[\tparam]{\ln \frac{\cdensstat{\chunk{X}{0}{T}}{\chunk{Y}{0}{T}}[\tparam]}{\cdens[\mu]{\chunk{X}{0}{T}}{\chunk{Y}{0}{T}}[\param_T]}} =\PEstat[\tparam]{
\idotsint
\left(\ln \frac{\cdensstat{\chunk{x}{0}{T}}{\chunk{Y}{0}{T}}[\tparam]}{\cdens[\mu]{\chunk{x}{0}{T}}{\chunk{Y}{0}{T}}[\param_T]}\right)
\cdensstat{\chunk{x}{0}{T}}{\chunk{Y}{0}{T}}[\tparam]
\prod_{i=0}^T \lambda(\rmd x_i)
} \geq 0
\]
because this quantity is the expectation under the stationary distribution of a Kullback-Leibler divergence.
This implies that
\begin{equation*}
\PEstat[\tparam]{\ln \frac{\densstat{\chunk{Y}{0}{T}}[\tparam]}{\dens[\mu]{\chunk{Y}{0}{T}}[\param_T]}} \leq \PEstat[\tparam]{\ln \frac{\densstat{\chunk{Y}{0}{T}}[\tparam]}{\dens[\mu]{\chunk{Y}{0}{T}}[\param_T]}}+\PEstat[\tparam]{\ln \frac{\cdensstat{\chunk{X}{0}{T}}{\chunk{Y}{0}{T}}[\tparam]}{\cdens[\mu]{\chunk{X}{0}{T}}{\chunk{Y}{0}{T}}[\param_T]}}= \PEstat[\tparam]{\ln \frac{\densstat{\chunk{X}{0}{T},\chunk{Y}{0}{T}}[\tparam]}{\dens[\mu]{\chunk{X}{0}{T},\chunk{Y}{0}{T}}[\param_T]}}.
\end{equation*}
On the other hand, using
\begin{multline*}
\PEstat[\tparam]{\ln \frac{\densstat{\chunk{X}{0}{T},\chunk{Y}{0}{T}}[\tparam]}{\dens[\mu]{\chunk{X}{0}{T},\chunk{Y}{0}{T}}[\param_T]}} \\=
\PEstat[\tparam]{\ln \left(\frac{ \pi^{\tparam}(X_0) g^{\tparam}(X_0,Y_0) }{\mu(X_0) g^{\param_T}(X_0,Y_0)} \right)}
+T \PEstat[\tparam]{\ln\left(\frac{m^{\tparam}(X_0,X_1)g^{\tparam}(X_1,Y_1) }
{ m^{\param_T}(X_0,X_1)g^{\param_T}(X_1,Y_1)}\right)}\eqsp,
\end{multline*}
we obtain, under \eqref{eq:nuit-de-singapour} and \eqref{eq:nuit-de-singapour-2}, that
\begin{equation}
\label{eq:bound-sup-T-exp}
\sup_{T \geq 0}\PEstat[\tparam]{\ln \frac{\densstat{\chunk{Y}{0}{T}}[\tparam]}{\dens[\mu]{\chunk{Y}{0}{T}}[\param_T]}}<\infty\eqsp.
\end{equation}
Assume first that
\begin{equation}\label{eq:two}
\limsup_{M  \to \infty}  \sup_{T \geq 0}\PPdoup[\mu]{\param_T}{\mineps{T}{N_T}^{-1}(\param_T)>M}=0\eqsp.
\end{equation}
The proof of the tightness of $\{\mineps{T}{N_T}^{-1}(\param_T)\}_{T \geq0}$ then follows by plugging \eqref{eq:bound-sup-T-exp} into \eqref{eq:one} and by noting that \eqref{eq:one} holds for all $\rho>\rme$, combined with \eqref{eq:two}.
To complete the proof, it thus remains to show \eqref{eq:two}. Rewriting the definition \eqref{eq:def-epsilon}, we obtain
\begin{equation*}
\mineps{T}{N_T}^{-1}(\param_T) \leq
\prod_{t=0}^T \frac{2 B_t^{\param_T}+N_T-2}{N_T-1} = \exp\left\{\sum_{t=0}^T \ln\left(\frac{2 B_t^{\param_T}+N_T-2}{N_T-1}\right)\right\}
\leq \exp\left\{\sum_{t=0}^T \frac{2 B_t^{\param_T}-1}{N_T-1}\right\}  \eqsp.
\end{equation*}
where $B_t^\param\eqdef \sup_{\ell \geq 0} \Bhat[t]{t+\ell}{\mu}{\param}$. This implies that for $M>1$, 
$$
\PPdoup[\mu]{\param_T}{\mineps{T}{N_T}^{-1}(\param_T)>M} \leq \PPdoup[\mu]{\param_T}{\sum_{t=0}^T \frac{2 B_t^{\param_T}-1}{N_T-1}>\ln M}
\leq \frac{1}{(\ln M)^\gamma} \PE[\param_T]{\left(\sum_{t=0}^T \frac{2 B_t^{\param_T}-1}{N_T-1}\right)^\gamma}[\mu] \eqsp.
$$
The proof of \eqref{eq:two} follows by noting that $N_T\sim T^{1/\gamma}$ and by using
$$
\PE[\param_T]{\left(\frac{\sum_{t=0}^T 2B_t^{\param_T}-1}{T^{1/\gamma}} \right)^\gamma}[\mu]
\leq \PE[\param_T]{\frac{\sum_{t=0}^T (2B_t^{\param_T})^\gamma}{T}}[\mu] \leq 2^\gamma \sup_{t \geq 0} \sup_{\param  \in \Param}\PE[\param]{(B_{t}^\param)^\gamma}[\mu]<\infty \eqsp,
$$
where the last inequality follows from \autoref{lem:tilde-B-moment-alpha}.
\end{proof}

\section*{Author address}
\noindent
{Fredrik Lindsten \\
Division of Automatic Control \\
Link\"oping University \\
SE--581 83 Link\"oping, Sweden\\
E-mail: \texttt{lindsten@isy.liu.se}\par}

\bibliographystyle{Chicago}
\bibliography{references}

\begin{thebibliography}{}

\bibitem[\protect\citeauthoryear{Andrieu, Doucet, and Holenstein}{Andrieu
  et~al.}{2010}]{AndrieuDH:2010}
Andrieu, C., A.~Doucet, and R.~Holenstein (2010).
\newblock Particle {M}arkov chain {M}onte {C}arlo methods.
\newblock {\em Journal of the Royal Statistical Society: Series B\/}~{\em
  72\/}(3), 269--342.

\bibitem[\protect\citeauthoryear{Andrieu, Lee, and Vihola}{Andrieu
  et~al.}{2013}]{AndrieuLV:2013}
Andrieu, C., A.~Lee, and M.~Vihola (2013, December).
\newblock Uniform ergodicity of the iterated conditional {SMC} and geometric
  ergodicity of particle {G}ibbs samplers.
\newblock Preprint, arXiv:1312.6432.

\bibitem[\protect\citeauthoryear{Andrieu and Roberts}{Andrieu and
  Roberts}{2009}]{AndrieuR:2009}
Andrieu, C. and G.~O. Roberts (2009).
\newblock The pseudo-marginal approach for efficient {M}onte {C}arlo
  computations.
\newblock {\em The Annals of Statistics\/}~{\em 37\/}(2), 697--725.

\bibitem[\protect\citeauthoryear{Andrieu and Vihola}{Andrieu and
  Vihola}{2012}]{AndrieuV:2012}
Andrieu, C. and M.~Vihola (2012, October).
\newblock Convergence properties of pseudo-marginal {M}arkov chain {M}onte
  {C}arlo algorithms.
\newblock arXiv.org, arXiv:1210.1484.

\bibitem[\protect\citeauthoryear{Beaumont}{Beaumont}{2003}]{beaumont:2003}
Beaumont, M.~A. (2003).
\newblock Estimation of population growth or decline in genetically monitored
  populations.
\newblock {\em Genetics\/}~{\em 164\/}(3), 1139--1160.

\bibitem[\protect\citeauthoryear{Capp\'e, Moulines, and Ryd\'en}{Capp\'e
  et~al.}{2005}]{CappeMR:2005}
Capp\'e, O., E.~Moulines, and T.~Ryd\'en (2005).
\newblock {\em Inference in Hidden {M}arkov Models}.
\newblock Springer.

\bibitem[\protect\citeauthoryear{Chopin and Singh}{Chopin and
  Singh}{2013}]{ChopinS:2013}
Chopin, N. and S.~S. Singh (2013, April).
\newblock On the particle {G}ibbs sampler.
\newblock Preprint, arXiv:1304.1887.

\bibitem[\protect\citeauthoryear{Del~Moral}{Del~Moral}{2004}]{delmoral:2004}
Del~Moral, P. (2004).
\newblock {\em {F}eynman-{K}ac Formulae - Genealogical and Interacting Particle
  Systems with Applications}.
\newblock Probability and its Applications. Springer.

\bibitem[\protect\citeauthoryear{Del~Moral and Doucet}{Del~Moral and
  Doucet}{2009}]{DelMoralD:2009}
Del~Moral, P. and A.~Doucet (2009).
\newblock {Particle methods: An introduction with applications}.
\newblock Research Report RR-6991, INRIA.

\bibitem[\protect\citeauthoryear{Del~Moral and Guionnet}{Del~Moral and
  Guionnet}{1999}]{DelMoralG:1999}
Del~Moral, P. and A.~Guionnet (1999).
\newblock Central limit theorem for nonlinear filtering and interacting
  particle systems.
\newblock {\em Annals of Applied Probability\/}~{\em 9\/}(2), 275--297.

\bibitem[\protect\citeauthoryear{Doucet, Godsill, and Andrieu}{Doucet
  et~al.}{2000}]{DoucetGA:2000}
Doucet, A., S.~J. Godsill, and C.~Andrieu (2000).
\newblock On sequential {M}onte {C}arlo sampling methods for {B}ayesian
  filtering.
\newblock {\em Statistics and Computing\/}~{\em 10\/}(3), 197--208.

\bibitem[\protect\citeauthoryear{Doucet and Johansen}{Doucet and
  Johansen}{2011}]{DoucetJ:2011}
Doucet, A. and A.~Johansen (2011).
\newblock A tutorial on particle filtering and smoothing: Fifteen years later.
\newblock In D.~Crisan and B.~Rozovskii (Eds.), {\em The Oxford Handbook of
  Nonlinear Filtering}. Oxford University Press.

\bibitem[\protect\citeauthoryear{Doucet, Pitt, and Kohn}{Doucet
  et~al.}{2012}]{DoucetPK:2012}
Doucet, A., M.~K. Pitt, and R.~Kohn (2012, October).
\newblock Efficient implementation of {M}arkov chain {M}onte {C}arlo when using
  an unbiased likelihood estimator.
\newblock Preprint, arXiv:1210.1871.

\bibitem[\protect\citeauthoryear{Fearnhead, Wyncoll, and Tawn}{Fearnhead
  et~al.}{2010}]{FearnheadWT:2010}
Fearnhead, P., D.~Wyncoll, and J.~Tawn (2010).
\newblock A sequential smoothing algorithm with linear computational cost.
\newblock {\em Biometrika\/}~{\em 97\/}(2), 447--464.

\bibitem[\protect\citeauthoryear{Godsill, Doucet, and West}{Godsill
  et~al.}{2004}]{GodsillDW:2004}
Godsill, S.~J., A.~Doucet, and M.~West (2004, March).
\newblock {M}onte {C}arlo smoothing for nonlinear time series.
\newblock {\em Journal of the American Statistical Association\/}~{\em
  99\/}(465), 156--168.

\bibitem[\protect\citeauthoryear{Golightly and Wilkinson}{Golightly and
  Wilkinson}{2011}]{GolightlyW:2011}
Golightly, A. and D.~J. Wilkinson (2011).
\newblock Bayesian parameter inference for stochastic biochemical network
  models using particle {M}arkov chain {M}onte {C}arlo.
\newblock {\em Interface Focus\/}~{\em 1\/}(6), 807--820.

\bibitem[\protect\citeauthoryear{Handschin and Mayne}{Handschin and
  Mayne}{1969}]{HandschinM:1969}
Handschin, J. and D.~Mayne (1969, May).
\newblock {M}onte {C}arlo techniques to estimate the conditional expectation in
  multi-stage non-linear filtering.
\newblock {\em International Journal of Control\/}~{\em 9\/}(5), 547--559.

\bibitem[\protect\citeauthoryear{Hull and White}{Hull and
  White}{1987}]{hull:white:1987}
Hull, J. and A.~White (1987).
\newblock The pricing of options on assets with stochastic volatilities.
\newblock {\em J. Finance\/}~{\em 42}, 281--300.

\bibitem[\protect\citeauthoryear{Jacquier, Polson, and Rossi}{Jacquier
  et~al.}{1994}]{jacquier:polson:rossi:1994}
Jacquier, E., N.~G. Polson, and P.~E. Rossi (1994).
\newblock {B}ayesian analysis of stochastic volatility models (with
  discussion).
\newblock {\em J. Bus. Econom. Statist.\/}~{\em 12}, 371--417.

\bibitem[\protect\citeauthoryear{Lee and Latuszynski}{Lee and
  Latuszynski}{2012}]{lee:latuszynski:2012}
Lee, A. and K.~Latuszynski (2012, October).
\newblock Variance bounding and geometric ergodicity of {M}arkov chain {M}onte
  {C}arlo kernels for approximate {B}ayesian computation.
\newblock Preprint, arXiv:1210.6703.

\bibitem[\protect\citeauthoryear{Lindsten, Jordan, and Sch\"on}{Lindsten
  et~al.}{2012}]{LindstenJS:2012}
Lindsten, F., M.~I. Jordan, and T.~B. Sch\"on (2012).
\newblock Ancestor sampling for particle {G}ibbs.
\newblock In P.~Bartlett, F.~C.~N. Pereira, C.~J.~C. Burges, L.~Bottou, and
  K.~Q. Weinberger (Eds.), {\em Advances in Neural Information Processing
  Systems ({NIPS}) 25}, pp.\  2600--2608.

\bibitem[\protect\citeauthoryear{Lindsten and Sch\"on}{Lindsten and
  Sch\"on}{2013}]{LindstenS:2013}
Lindsten, F. and T.~B. Sch\"on (2013).
\newblock Backward simulation methods for {M}onte {C}arlo statistical
  inference.
\newblock {\em Foundations and Trends in Machine Learning\/}~{\em 6\/}(1),
  1--143.

\bibitem[\protect\citeauthoryear{Meyn and Tweedie}{Meyn and
  Tweedie}{2009}]{MeynT:2009}
Meyn, S. and R.~L. Tweedie (2009).
\newblock {\em Markov Chains and Stochastic Stability\/} (2nd ed.).
\newblock Cambridge University Press.

\bibitem[\protect\citeauthoryear{Pitt, Silva, Giordani, and Kohn}{Pitt
  et~al.}{2012}]{PittSGK:2012}
Pitt, M.~K., R.~S. Silva, P.~Giordani, and R.~Kohn (2012).
\newblock On some properties of {M}arkov chain {M}onte {C}arlo simulation
  methods based on the particle filter.
\newblock {\em Journal of Econometrics\/}~{\em 171}, 134--151.

\bibitem[\protect\citeauthoryear{Rasmussen, Ratmann, and Koelle}{Rasmussen
  et~al.}{2011}]{RasmussenRK:2011}
Rasmussen, D.~A., O.~Ratmann, and K.~Koelle (2011).
\newblock Inference for nonlinear epidemiological models using genealogies and
  time series.
\newblock {\em PLoS Comput Biology\/}~{\em 7\/}(8).

\bibitem[\protect\citeauthoryear{Rubin}{Rubin}{1987}]{Rubin:1987}
Rubin, D.~B. (1987, June).
\newblock A noniterative sampling/importance resampling alternative to the data
  augmentation algorithm for creating a few imputations when fractions of
  missing information are modest: The {SIR} algorithm.
\newblock {\em Journal of the American Statistical Association\/}~{\em
  82\/}(398), 543--546.
\newblock Comment to Tanner and Wong: The Calculation of Posterior
  Distributions by Data Augmentation.

\bibitem[\protect\citeauthoryear{Shephard and Andersen}{Shephard and
  Andersen}{2009}]{shephard:andersen;2009}
Shephard, N. and T.~Andersen (2009).
\newblock Stochastic volatility: {O}rigins and overview.
\newblock In T.~Andersen, R.~Davis, J.-P. Kreiss, and T.~Mikosch (Eds.), {\em
  Handbook of Financial Time Series}. Berlin: Springer.

\bibitem[\protect\citeauthoryear{Taylor}{Taylor}{1982}]{taylor:1982}
Taylor, S.~J. (1982).
\newblock Financial returns modelled by the product of two stochastic processes
  -- {A} study of daily sugar prices, 1961-79.
\newblock In O.~D. Anderson (Ed.), {\em Time Series Analysis: Theory and
  Practice}, Volume~1, pp.\  203--226. Elsevier/North-Holland, Amsterdam.

\bibitem[\protect\citeauthoryear{Vrugt, ter Braak, Diks, and Schoups}{Vrugt
  et~al.}{2013}]{VrugtBDS:2013}
Vrugt, J.~A., J.~F. ter Braak, C.~G.~H. Diks, and G.~Schoups (2013).
\newblock Hydrologic data assimilation using particle {M}arkov chain {M}onte
  {C}arlo simulation: {T}heory, concepts and applications.
\newblock {\em Advances in Water Resources\/}~{\em 51}, 457--478.

\bibitem[\protect\citeauthoryear{Whiteley, Andrieu, and Doucet}{Whiteley
  et~al.}{2010}]{WhiteleyAD:2010}
Whiteley, N., C.~Andrieu, and A.~Doucet (2010).
\newblock Efficient {B}ayesian inference for switching state-space models using
  discrete particle {M}arkov chain {M}onte {C}arlo methods.
\newblock Technical report, Bristol Statistics Research Report 10:04.

\end{thebibliography}

\end{document}